%% file: roosta_ascher_arxiv.tex
\documentclass[12pt]{article}

\usepackage{graphicx}
\usepackage{amsmath}
\usepackage{amssymb}
\usepackage{latexsym}
\usepackage{subfigure}
\usepackage{crop}
\usepackage{algorithmic}
\usepackage{algorithm}
\usepackage{multirow}
\usepackage[section,subsection,subsubsection]{placeins}
\usepackage[colorlinks = true, pdfstartview = FitV, linkcolor = blue, citecolor = blue, urlcolor = blue]{hyperref}
\usepackage{epstopdf}

\newcommand {\zz}  { {\bf z} }

\newcommand {\xx}  { {\bf x} }

\newcommand{\R}{{\rm I\!R}}

\newcommand {\KH}  { {\mathcal K}_H }
\newcommand {\KG}  { {\mathcal K}_G }
\newcommand {\KU}  { {\mathcal K}_U }

\newcommand {\veps} {\varepsilon}

\newcommand {\mm}  { {\bf m} }

\newcommand {\ww}  { {\bf w} }

\newcommand {\dd}  { {\bf d} }
\newcommand {\ee}  { {\bf e} }

\newcommand{\defeq}{\mathrel{\mathop:}=}

\newtheorem{theorem}{Theorem}

\newtheorem{corollary}[theorem]{Corollary}

\newtheorem{lemma}[theorem]{Lemma}
\newtheorem{example}{Example}

\newenvironment{proof}[1][Proof]{\begin{trivlist}
\item[\hskip \labelsep {\bfseries #1}]}{\end{trivlist}}

\newcommand{\qed}{\nobreak \ifvmode \relax \else
      \ifdim\lastskip<1.5em \hskip-\lastskip
      \hskip1.5em plus0em minus0.5em \fi \nobreak
      \vrule height0.75em width0.5em depth0.25em\fi}

\begin{document}

\title{Improved bounds on sample size for implicit matrix trace estimators}

\author{Farbod Roosta-Khorasani and Uri Ascher
\thanks{Dept. of Computer Science, University of British Columbia, Vancouver, Canada
{\tt farbod/ascher@cs.ubc.ca} .
This work was supported in part by NSERC Discovery Grant 84306.}}

\maketitle

\begin{abstract}

This article is concerned with Monte-Carlo methods for the estimation of the trace of
an implicitly given matrix $A$ whose information is only available through matrix-vector products.
Such a method approximates the trace by an average of $N$ expressions of the form
$\ww^t (A\ww )$, with random vectors $\ww$ drawn from an appropriate distribution. 
We prove, discuss and experiment with bounds on the number of realizations $N$ required
in order to guarantee a probabilistic bound on the relative error of the trace estimation
upon employing 
Rademacher (Hutchinson), Gaussian and 
uniform unit vector (with and without replacement) probability distributions.

In total, one necessary bound and six sufficient bounds are proved,
improving upon and extending similar estimates obtained in the seminal work of Avron and Toledo (2011)
in several dimensions. 
We first improve their bound on $N$ for the Hutchinson method, 
dropping a term that relates to $rank(A)$ and making the bound comparable
with that for the Gaussian estimator.

We further prove new sufficient bounds for the Hutchinson, Gaussian and the unit vector estimators,
as well as a necessary bound for the Gaussian estimator,
which depend more specifically on properties of the matrix $A$. As such they may
suggest for what type of matrices one distribution or another provides a particularly
effective or relatively ineffective stochastic estimation method.

\end{abstract}


\vspace{1pc}

\vspace{0.5pc} \noindent



\section{Introduction}
\label{sec:int}
The need to estimate the trace of an implicit square matrix
is of fundamental importance~\cite{sdr} and arises in many applications; 
see for instance \cite{hutchinson,bafago,avto,HaberChungHermann2010,doas3,yori,rodoas1,learhe,gohewa,avron}
and references therein.
By ``implicit'' we mean that 
the matrix of interest is not available explicitly: only probes in the form of matrix-vector products for any appropriate vector are available.
The standard approach for estimating the trace of such a matrix $A$ is based on a Monte-Carlo method, where one generates $N$ 
random vector realizations $\ww_{i}$ from a suitable probability distribution $D$ and computes 
\begin{equation}
tr_{D}^{N}(A) \defeq \frac{1}{N} \sum_{i=1}^{N} \ww_{i}^{t} A \ww_{i} .
\label{tr_moncar}
\end{equation}
For the popular case where $A$ is symmetric positive semi-definite (SPSD),
the original method for estimating its trace, $tr(A)$,
is due to Hutchinson~\cite{hutchinson} and uses the Rademacher distribution for $D$. 

Until the work by Avron and Toledo~\cite{avto}, the main analysis and comparison of such methods was based on the variance of one sample. 
It is known that compared to other methods the Hutchinson method has the smallest variance,  
and as such it has been extensively used in many applications. 
In~\cite{avto} so-called $(\veps,\delta)$ bounds are derived in which, using Chernoff-like analysis, 
a lower bound is obtained on the number of samples required to achieve a probabilistically guaranteed relative error of the estimated trace. 
More specifically, for a given pair $(\veps,\delta)$ of small (say, $< 1$) positive values and an appropriate probability distribution $D$, 
a lower bound on $N$ is provided such that
\begin{equation}
Pr\left(| tr_{D}^{N}(A) - tr(A) | \leq \veps~ tr(A) \right) \geq 1-\delta .
\label{prob_tr}
\end{equation}
%
These authors further suggest that minimum-variance estimators may not be practically best, 
and conclude based on their analysis that the method with the best bound is the one using the Gaussian distribution.
Let us denote
\begin{subequations}
\begin{eqnarray}
c &=& c( \veps,\delta) \defeq \veps^{-2} \ln (2/\delta), \label{common_factor_c} \\
r &=& rank(A). \label{common_factor_r}
\end{eqnarray}
\label{common_factor}
\end{subequations}
Then \cite{avto} showed that, provided $A$ is real SPSD, \eqref{prob_tr} holds for the Hutchinson method if
$N \geq 6(c+ \veps^{-2}\ln r )$ and for the Gaussian distribution if  $N \geq 20c$.

In the present paper we continue to consider the same objective as in~\cite{avto}, and our first task is to improve on these bounds.
Specifically, in Theorems~\ref{hutch_thm_01} and~\ref{gauss_thm_01}
we show that \eqref{prob_tr} holds 
for the Hutchinson method if
\begin{eqnarray}
N \geq 6 c (\veps,\delta), \label{hutch_bd_01}
\end{eqnarray}
and for the Gaussian distribution if  
\begin{eqnarray}
N \geq 8 c(\veps,\delta). \label{gauss_bd_01}
\end{eqnarray}
The bound~\eqref{hutch_bd_01} removes a previous factor involving
the rank of the matrix $A$, conjectured in~\cite{avto} to be indeed redundant. 
Note that these two bounds are astoundingly simple and general: they hold for any SPSD matrix,
regardless of size or any other matrix property. Thus, we cannot expect them to be tight in practice
 for many specific instances of $A$ that arise in applications.

Although practically useful, the bounds on $N$ given in~\eqref{hutch_bd_01} and \eqref{gauss_bd_01}
do not provide insight into how different types of matrices are handled with each probability distribution. 
Our next contribution is to provide different bounds for the Gaussian and Hutchinson trace estimators 
which, though generally not computable for implicit matrices, do shed light on this question.

Furthermore, for the Gaussian estimator we prove a practically useful {\em necessary lower} bound on $N$,
for a given pair $(\veps, \delta)$.

A third probability distribution we consider was called
the unit vector distribution in~\cite{avto}. Here, 
the vectors $\ww_{i}$ in~\eqref{tr_moncar} are uniformly drawn from the columns of a scaled identity matrix, $\sqrt{n} I$,
and $A$ need not be SPSD. 
We slightly generalize the bound in~\cite{avto}, obtained for the case where the sampling is done with replacement.
Our bound, although not as simply computed as \eqref{hutch_bd_01} or \eqref{gauss_bd_01}, 
can be useful in determining
which types of matrices this distribution works best on. 
We then give a tighter bound for the case where the sampling is done without replacement, 
suggesting that when the difference between the bounds is significant 
(which happens when $N$ is large),
a uniform random sampling of unit vectors without replacement may be a more advisable distribution to estimate the trace with.

This paper is organized as follows.
Section~\ref{hutch} gives two bounds for the Hutchinson method as advertised above,
namely the improved bound~\eqref{hutch_bd_01}
and a more involved but potentially more informative bound. 
Section~\ref{gauss} deals likewise
with the Gaussian method and adds also a necessary lower bound, while  
Section~\ref{randsamp} is devoted to the unit vector sampling methods.

In Section~\ref{numer} we give some numerical examples verifying that the trends predicted by the theory are indeed realized. Conclusions and further thoughts are gathered in Section~\ref{concl}.

In what follows we use the notation $tr_{H}^{N}(A)$, $tr_{G}^{N}(A)$, $tr_{U_{1}}^{N}(A)$, and $tr_{U_{2}}^{N}(A)$ to refer, respectively, 
to the trace estimators using Hutchinson, Gaussian, 
and uniform unit vector with and without replacement, in lieu of the generic notation $tr_{D}^{N}(A)$ 
in~\eqref{tr_moncar} and \eqref{prob_tr}.
We also denote for any given random vector of size $n$, $\ww_i = (w_{i1}, w_{i2}, \ldots , w_{in})^t$. We restrict attention to real-valued matrices, although
extensions to complex-valued ones are possible, and employ the $2$-norm by default. 


\section{Hutchinson estimator bounds}
\label{hutch}
In this section we consider the Hutchinson trace estimator, $tr_{H}^{N}(A)$, obtained by setting $D=H$ in~\eqref{tr_moncar}, 
where the components of the random vectors $\ww_{i}$ are i.i.d Rademacher random variables 
(i.e., $Pr(w_{ij} = 1) = Pr(w_{ij} = -1) = \frac{1}{2}$). 

\subsection{Improving the bound in~\cite{avto}}
\label{hutch1}

\begin{theorem}
Let $A$ be an $n \times n$ SPSD matrix. Given a pair $(\veps,\delta)$, 
the inequality ~\eqref{prob_tr} holds with $D=H$ if $N$ satisfies~\eqref{hutch_bd_01}.
\label{hutch_thm_01}
\end{theorem}
\begin{proof}
Since $A$ is SPSD, it can be diagonalized by a unitary similarity transformation as $A = U^{t} \Lambda U$. 
Consider $N$ random vectors $\ww_{i}, \; i = 1, \ldots , N$, whose components are i.i.d and drawn from the Rademacher distribution,
and define $\zz_{i} = U \ww_{i}$ for each. We have

\begin{eqnarray*}
Pr\left( tr_{H}^{N}(A) \leq (1-\veps)tr(A) \right) &=& Pr\left( \frac{1}{N} \sum_{i=1}^{N} \ww_{i}^{t} A \ww_{i}  \leq (1-\veps)tr(A)\right) 
\nonumber \\
&=& Pr\left( \frac{1}{N} \sum_{i=1}^{N} \zz_{i}^{t} \Lambda \zz_{i}  \leq (1-\veps)tr(A)\right) \nonumber \\
&=& Pr\left( \sum_{i=1}^{N} \sum_{j=1}^{r} \lambda_{j} z_{ij}^{2}  \leq N(1-\veps)tr(A)\right) \nonumber \\
&=& Pr\left( \sum_{j=1}^{r} \frac{\lambda_{j}}{tr(A)} \sum_{i=1}^{N} z_{ij}^{2} \leq N(1-\veps)\right) \nonumber \\
&\leq& \exp \{t N (1-\veps) \} \mathbb{E}\left( \exp\{\sum_{j=1}^{r} \frac{\lambda_{j}}{tr(A)} \sum_{i=1}^{N} -t z_{ij}^{2}\}\right),
\end{eqnarray*}
where the last inequality holds for any $t >0$ by Markov's inequality. 

Next, using the convexity of the $\exp$ function and the linearity of expectation, we obtain
\begin{eqnarray*}
\mathbb{E}\left(\exp\{\sum_{j=1}^{r} \frac{\lambda_{j}}{tr(A)} \sum_{i=1}^{N} -t z_{ij}^{2}\}\right) &\leq& 
\sum_{j=1}^{r} \frac{\lambda_{j}}{tr(A)} \mathbb{E}\left(\exp\{\sum_{i=1}^{N} -t z_{ij}^{2}\}\right) \\
&=& 
\sum_{j=1}^{r} \frac{\lambda_{j}}{tr(A)} \mathbb{E}\left( \prod_{i=1}^{N} \exp\{-t z_{ij}^{2}\}\right) \\
&=& 
\sum_{j=1}^{r} \frac{\lambda_{j}}{tr(A)} \prod_{i=1}^{N}\mathbb{E}\left(\exp\{-t z_{ij}^{2}\}\right),
\end{eqnarray*}
where the last equality holds since, for a given $j$, $z_{ij}$'s are independent with respect to $i$. 

Now, we want to have that $\exp\{t N (1-\veps) \} \prod_{i=1}^{N}\mathbb{E}\left( \exp\{-t z_{ij}^{2}\}\right) \leq \delta / 2$. 
For this we make use of the inequalities in the end of the proof of Lemma~5.1 of~\cite{achlioptas}. 
Following inequalities (15)--(19) in~\cite{achlioptas} and letting $t = \veps / (2(1+\veps))$, we get
\begin{eqnarray*}
\exp\{t N (1-\veps) \} \prod_{i=1}^{N}\mathbb{E}\left(\exp\{-t z_{ij}^{2}\}\right) < \exp\{-\frac{N}{2}(\frac{\veps^2}{2}-\frac{\veps^3}{3})\}.
\end{eqnarray*}
Next, if $N$ satisfies \eqref{hutch_bd_01} then
$\exp\{-\frac{N}{2}(\frac{\veps^2}{2}-\frac{\veps^3}{3})\} < \delta/2$, and thus it follows that 
\begin{equation*}
Pr \left(tr_{H}^{N}(A) \leq (1-\veps)tr(A)\right) < \delta/2 .
\end{equation*}

By a similar argument, making use of inequalities (11)--(14) in~\cite{achlioptas} with the same $t$ as above, 
we also obtain with the same bound for $N$ so that $Pr\left(tr_{H}^{N}(A) \geq (1+\veps)tr(A)\right) \leq \delta/2$. 
So finally using the union bound yields the desired result.
$\blacksquare$
\end{proof}

It can be seen that \eqref{hutch_bd_01} is the same bound as the one in~\cite{avto} 
with the important exception that the factor $r=rank(A)$ does not appear in the bound.
Furthermore, the same bound on $N$ holds for any SPSD matrix.


\subsection{A matrix-dependent bound}
\label{hutch2}

Here we derive another bound for
the Hutchinson trace estimator 
which may shed light
as to what type of matrices the Hutchinson method is best suited for. 

Let us denote by $a_{k,j}$ the $(k,j)$th element of $A$ and by ${\bf a}_j$ its $j$th column, 
$k, j = 1, \ldots , n$.

\begin{theorem}
Let $A$ be an $n \times n$ symmetric positive semi-definite matrix, and define
\begin{equation}
\KH^{j} \defeq \frac{\| {\bf a}_j \|^{2} - a_{j,j}^{2}}{a_{j,j}^{2}}  =  \sum_{k \neq j} a_{k,j}^2 ~/~ a_{j,j}^2,
\quad \KH \defeq \max_{j} \KH^{j} .
\label{hutch_energy_dist}
\end{equation}
Given a pair of positive small values
$(\veps,\delta)$, the inequality~\eqref{prob_tr} holds with $D=H$ if 
\begin{equation}
N > 2 \KH c(\veps,\delta) .
\label{hutch_bd_02}
\end{equation}

\label{hutch_thm_02}
\end{theorem}
\begin{proof}
Elementary linear algebra implies that since $A$ is SPSD, $a_{j,j} \geq 0$ for each $j$. Furthermore, if $a_{j,j} = 0$ then
the $j$th row and column of $A$ identically vanish,
so we may assume below that $a_{j,j} > 0$ for all $j = 1, \ldots , n$.
Note that 
$$tr_{H}^{N}(A) - tr(A) = \frac{1}{N} \sum_{j=1}^{n} \sum_{i=1}^{N}  \sum_{\substack{k=1 \\ k\neq j}}^{n} a_{j,k}w_{ij} w_{ik}.$$ Hence
\begin{eqnarray*}
Pr\left(tr_{H}^{N}(A) \leq (1-\veps)tr(A)\right) 
&=& Pr\left(\sum_{j=1}^{n} \sum_{i=1}^{N}  \sum_{\substack{k=1 \\ k\neq j}}^{n} -a_{j,k} w_{ij} w_{ik}  \geq N \veps~ tr(A)\right) \\
&=& Pr\left(\sum_{j=1}^{n} \frac{a_{j,j}}{tr(A)} \sum_{i=1}^{N} \sum_{\substack{k=1 \\ k\neq j}}^{n} -\frac{a_{j,k}}{a_{j,j}} w_{ij} w_{ik}  \geq N \veps\right) \\
&\leq& \exp\{-t N \veps\} \mathbb{E}\left(\exp \{ \sum_{j=1}^{n} \frac{a_{j,j}}{tr(A)} \sum_{i=1}^{N} \sum_{\substack{k=1 \\ k\neq j}}^{n} -\frac{a_{j,k} t}{a_{j,j}} w_{ij} w_{ik}\}\right),
\end{eqnarray*}
where the last inequality is again obtained for any $t > 0$ by using Markov's inequality. 
Now, again using the convexity of the $\exp$ function and the linearity of expectation, we obtain
\begin{eqnarray*}
Pr\left(tr_{H}^{N}(A) \leq (1-\veps)tr(A)\right) &\leq& \exp\{-t N \veps\} \sum_{j=1}^{n} \frac{a_{j,j}}{tr(A)} 
\mathbb{E}\left(\exp \{ \sum_{i=1}^{N} \sum_{\substack{k=1 \\ k\neq j}}^{n} -\frac{a_{j,k} t}{a_{j,j}} w_{ij} w_{ik}\}\right)  \\
&=& \exp\{-t N \veps\} \sum_{j=1}^{n} \frac{a_{j,j}}{tr(A)} \prod_{i=1}^{N} \mathbb{E}\left(\exp \{\sum_{\substack{k=1 \\ k\neq j}}^{n} -\frac{a_{j,k} t}{a_{j,j}} w_{ij} w_{ik}\}\right) 
\end{eqnarray*}
by independence of $w_{ij} w_{ik}$ with respect to the index $i$.

Next, note that 
\begin{eqnarray*}
\mathbb{E}\left(\exp \{\sum_{\substack{k=1 \\ k\neq j}}^{n} \frac{a_{j,k} t}{a_{j,j}} w_{ik} \}\right) = 
\mathbb{E}\left(\exp \{\sum_{\substack{k=1 \\ k\neq j}}^{n} -\frac{a_{j,k} t}{a_{j,j}} w_{ik} \}\right).
\end{eqnarray*}
Furthermore, since
$Pr( w_{ij} = -1) = Pr( w_{ij} = 1) = \frac{1}{2}$, and using the law of total expectation, 
we have
\begin{eqnarray*}
\mathbb{E}\left(\exp \{\sum_{\substack{k=1 \\ k\neq j}}^{n} -\frac{a_{j,k} t}{a_{j,j}} w_{ij} w_{ik} \}\right) =  
\mathbb{E}\left(\exp \{ \sum_{\substack{k=1 \\ k\neq j}}^{n} \frac{a_{j,k} t}{a_{j,j}} w_{ik} \}\right) 
= \prod_{\substack{k=1 \\ k\neq j}}^{n} \mathbb{E}\left(\exp \{\frac{a_{j,k} t}{a_{j,j}} w_{ik} \}\right) ,
\end{eqnarray*}
so
\begin{eqnarray*}
Pr\left(tr_{H}^{N}(A) \leq (1-\veps)tr(A)\right) &\leq& \exp \{-t N \veps\} \sum_{j=1}^{n} \frac{a_{j,j}}{tr(A)} \prod_{i=1}^{N} 
\prod_{\substack{k=1 \\ k\neq j}}^{n} \mathbb{E}\left(\exp \{\frac{a_{j,k} t}{a_{j,j}} w_{ik} \}\right) .
\end{eqnarray*}
We want to have the right hand side expression bounded by $\delta / 2$.


Applying Hoeffding's lemma we get
$\mathbb{E}\left(\exp \{ \frac{a_{j,k} t}{a_{j,j}} w_{ik} \}\right) \leq \exp\{ \frac{a_{j,k}^{2} t^{2}}{2 a_{j,j}^{2}} \}$,
hence
\begin{subequations}
\begin{eqnarray}
\exp \{-t N \veps\} \prod_{i=1}^{N} \prod_{\substack{k=1 \\ k\neq j}}^{n} \mathbb{E}\left(\exp \{ \frac{a_{j,k} t}{a_{j,j}} w_{ik} \}\right) &\leq& \exp \{-t N \veps + \KH^{j} N t^{2}/2 \} \label{K_H_j_dist}\\
&\leq& \exp \{ -t N \veps + \KH N t^{2}/2 \} \label{K_H_j_max}.
\end{eqnarray}
\end{subequations}
The choice 
\begin{eqnarray*}
t = \veps / \KH
\end{eqnarray*}
minimizes the right hand side. Now if \eqref{hutch_bd_02} holds then
\begin{eqnarray*}
\exp(-t N \veps) \prod_{i=1}^{N} \prod_{\substack{k=1 \\ k\neq j}}^{n} \mathbb{E}\left(\exp \{\frac{a_{j,k} t}{a_{j,j}} w_{ik} \}\right) \leq \delta / 2 ,
\end{eqnarray*}
hence we have
\begin{eqnarray*}
Pr(tr_{H}^{N}(A) \leq (1-\veps)tr(A)) &\leq& \delta / 2.
\end{eqnarray*}
Similarly, we obtain that
\begin{eqnarray*}
Pr(tr_{H}^{N}(A) \geq (1+\veps)tr(A)) &\leq& \delta/2,
\end{eqnarray*}
and using the union bound finally gives desired result.
$\blacksquare$
\end{proof}

Comparing \eqref{hutch_bd_02} to \eqref{hutch_bd_01}, it is clear that the bound of the present subsection
is only worthy of consideration if $\KH < 3$. 
Note that Theorem~\ref{hutch_thm_02} emphasizes the relative $\ell_2$ energy of the off-diagonals: 
the matrix does not necessarily have to be diagonally dominant 
(i.e., where a similar relationship holds in  the $\ell_1$ norm) 
for the bound on $N$ to be moderate.
Furthermore, a matrix need not be ``nearly'' diagonal for this method to require small sample size. In fact a matrix can have off-diagonal elements of significant size that are far away from the main diagonal
without automatically affecting the performance of the Hutchinson method.
However, note also that our bound can be pessimistic, especially if the average value or the mode of 
$\KH^{j}$ in \eqref{hutch_energy_dist} is far lower than its maximum, $\KH$.
This can be seen in the above proof where the estimate~\eqref{K_H_j_max} 
is obtained from~\eqref{K_H_j_dist}. Simulations in Section~\ref{numer} show that
the Hutchinson method can be a very efficient estimator
even in the presence of large outliers, 
so long as the bulk of the distribution is concentrated near  small values.  

The case $\KH=0$ corresponds to a diagonal matrix, for which the Hutchinson method yields
the trace with one shot, $N=1$.
In agreement with the bound \eqref{hutch_bd_02}, 
we expect the actual required $N$ to grow when a sequence of otherwise similar matrices $A$
is envisioned in which $\KH$ grows away from $0$, as the energy in the off-diagonal elements grows relatively
to that in the diagonal elements.

\section{Gaussian estimator bounds}
\label{gauss}

In this section we consider the Gaussian trace estimator, $tr_{G}^{N}(A)$, obtained by setting $D=G$ in~\eqref{tr_moncar}, 
where the components of the random vectors $\ww_{i}$ are i.i.d standard normal random variables. 
We give two sufficient and one necessary lower bounds for the number of Gaussian samples required to achieve an $(\veps, \delta)$ trace estimate. The first sufficient bound \eqref{gauss_bd_01}
improves the result in \cite{avto} by a factor of $2.5$.
Our bound is only worse than \eqref{hutch_bd_01} by a fraction,
and it is an upper limit
of the potentially more informative (if less available)
bound~\eqref{gauss_bd_02}, which relates to the properties of the matrix $A$. 
The bound~\eqref{gauss_bd_02} provides an indication as to what matrices may be suitable candidates for the Gaussian method. 
Then we present a practically computable, necessary bound for the sample size $N$.

\subsection{Sufficient bounds}
\label{gauss1}
The proof of the following theorem closely follows the approach in~\cite{avto}.

\begin{theorem}
Let $A$ be an $n \times n$ SPSD matrix
and denote its eigenvalues by $\lambda_1, \ldots , \lambda_n$. 
Further, define
\begin{equation}
\KG^{j} \defeq \frac{\lambda_{j}}{tr(A)} , \quad \KG \defeq \max_j \KG^{j} = \frac{\| A \|}{tr(A)} . 
\label{gauss_energy_dist}
\end{equation}
Then, given a pair of positive small
values $(\veps,\delta)$, the inequality~\eqref{prob_tr} holds with $D=G$ provided that 
\eqref{gauss_bd_01} holds.
This estimate also holds provided that
\begin{equation}
N > 8 \KG c(\veps,\delta) .
\label{gauss_bd_02}
\end{equation}
\label{gauss_thm_01}
\end{theorem}
\begin{proof}
Since $A$ is SPSD, we have $\| A \| \leq tr(A)$, so if \eqref{gauss_bd_01} holds then so does \eqref{gauss_bd_02}.
We next concentrate on proving the result assuming the tighter bound \eqref{gauss_bd_02}
on the actual $N$ required in a given instance.

Writing as in the previous section $A = U^{t} \Lambda U$, 
consider $N$ random vectors $\ww_{i}, \; i = 1, \ldots , N$, whose components are i.i.d and drawn from the normal distribution,
and define $\zz_{i} = U \ww_{i}$. Since $U$ is orthogonal, the elements $z_{ij}$ of $\zz_{i}$
are i.i.d Gaussian random variables. 
We have as before,
\begin{eqnarray*}
Pr \left( tr_{G}^{N}(A) \leq (1-\veps)tr(A)\right) 
&=& Pr\left(\sum_{i=1}^{N} \sum_{j=1}^{r} \lambda_{j} z_{ij}^{2}  \leq N(1-\veps)tr(A)\right)  \\
&\leq& \exp\{t N (1-\veps) tr(A)\} \mathbb{E}\left(\exp\{\sum_{i=1}^{N} \sum_{j=1}^{r} -t \lambda_{j} z_{ij}^{2}\}\right)  \\
&\leq& \exp\{t N (1-\veps) tr(A)\} \prod_{i=1}^{N} \prod_{j=1}^{r} \mathbb{E}\left(\exp\{ -t \lambda_{j} z_{ij}^{2}\}\right) .
\end{eqnarray*}

Here $z_{ij}^{2}$ is a $\chi^2$ random variable of degree 1 (see ~\cite{mogrbo}), and hence for the characteristics 
we have 
\begin{equation*}
\mathbb{E}\left(\exp\{ -t \lambda_{j} z_{ij}^{2}\}\right) = (1 + 2 \lambda_{j} t)^{-\frac{1}{2}} .
\end{equation*}
This yields the bound
\begin{equation*}
Pr\left(tr_{G}^{N}(A) \leq (1-\veps)tr(A)\right) \leq \exp\{t N (1-\veps) tr(A)\} \prod_{j=1}^{r} (1 + 2 \lambda_{j} t)^{-\frac{N}{2}} .
\end{equation*}

Next, it is easy to prove by elementary calculus that given any $0 < \alpha < 1$, the following holds for all $0 \leq x \leq \frac{1-\alpha}{\alpha}$,
\begin{equation}
\ln(1+x) \geq \alpha x .
\label{lemma2}
\end{equation}
Setting $\alpha = 1-\veps/2$, then by~\eqref{lemma2} and for all 
$t \leq (1-\alpha)/( 2 \alpha \| A \| )$, 
we have that $(1 + 2 \lambda_{j} t) > \exp \{2 \alpha \lambda_{j}\} t$, so
\begin{eqnarray*}
Pr\left(tr_{G}^{N}(A) \leq (1-\veps)tr(A)\right) &\leq& \exp \{t N (1-\veps) tr(A)\} \prod_{j=1}^{r} \exp( - N \alpha \lambda_{j} t) \\
&=& \exp\{t N (1 - \veps - \alpha) tr(A)\} .
\end{eqnarray*}

We want the latter right hand side to be bounded by $\delta /2$,
i.e., we want to have 
\begin{eqnarray*}
N &\geq& \frac{\ln\big(2 /\delta\big)}{(\alpha-(1-\veps)) tr(A) t} 
= \frac{2 \veps c(\veps,\delta)}{tr(A) t},
\end{eqnarray*}
where $t \leq (1-\alpha )/( 2 \alpha \| A \| )$. 
Now, setting $t = (1-\alpha)/( 2 \alpha \| A \|) = \veps/( 2 (2 - \veps) \| A \| )$, we obtain
\begin{equation*}
N \geq 4 (2-\veps) c(\veps,\delta)  \KG,
\end{equation*}
so if \eqref{gauss_bd_02} holds then
\begin{equation*}
Pr\left(tr_{G}^{N}(A) \leq (1-\veps)tr(A)\right) \leq \delta /2 .
\end{equation*}

Using a similar argument we also obtain
\begin{equation*}
Pr\left(tr_{G}^{N}(A) \geq (1+\veps)tr(A)\right) \leq \delta / 2 ,
\end{equation*}
and subsequently the union bound yields the desire result.
$\blacksquare$
\end{proof}

The matrix-dependent bound~\eqref{gauss_bd_02}, proved to be sufficient in Theorem~\ref{gauss_thm_01},
provides additional information over~\eqref{gauss_bd_01}
about the type of matrices for which the Gaussian estimator is (probabilistically) guaranteed 
to require only a small sample size: 
if the eigenvalues of an SPSD matrix are distributed such that the ratio $\|A\| / tr(A)$ is small
(e.g., if they are all of approximately the same size), 
then the Gaussian estimator bound requires a small number of realizations. This observation is reaffirmed by looking at the variance of this estimator, namely $2 \| A \|^{2}_{F}$. It is easy to show that among all the matrices with a fixed trace and rank, those with equal eigenvalues have the smallest Frobenius norm.

Furthermore, it is easy to see that the {\em stable rank} (see~\cite{tropp} and references therein)
of any real rectangular matrix $C$ which satisfies $A = C^t C$ equals $1 / \KG$.
Thus, the bound constant in~\eqref{gauss_bd_02} is inversely proportional to this stable rank,
suggesting that estimating the trace using the Gaussian distribution
may become inefficient if the stable rank of the matrix is low. 
Theorem~\ref{gauss_thm_02} in Section~\ref{gauss2} below further substantiates this intuition.

As an example of an application of the above results, let us consider finding
the minimum number of samples required to compute the rank of a projection matrix using the Gaussian estimator~\cite{avto,bekosa}. 
Recall that a projection matrix is SPSD with only 0 and 1 eigenvalues. 
Compared to the derivation in~\cite{avto}, 
here we use Theorem~\ref{gauss_thm_01} directly to obtain a similar bound with a slightly better constant. 

\begin{corollary}
Let $A$ be an $n \times n$ projection matrix with rank $r > 0$, and denote 
the rounding of any real scalar $x$ to the nearest integer by $round(x)$. 
Then, given a positive small value $\delta$, the estimate 
\begin{subequations}
\begin{eqnarray}
Pr\left(round(tr_{G}^{N}(A)) \neq r\right) \leq \delta
\label{11a}
\end{eqnarray}
holds if 
\begin{eqnarray}
N \geq 8 \; r \ln\big(2 /\delta\big) .
\label{11b}
\end{eqnarray}
\label{11}
\end{subequations}
\end{corollary}
\begin{proof}
The result immediately follows using Theorem~\ref{gauss_thm_01} upon setting $\veps = 1/r, 
\|A\| = 1$ and $tr(A) = r$. 
$\blacksquare$
\end{proof}


\subsection{A necessary bound}
\label{gauss2}

Below we provide a rank-dependent,
almost tight necessary condition for the minimum sample size required to obtain~\eqref{prob_tr}.
This bound is easily computable in case that $r=rank(A)$ is known. 

Before we proceed, recall the definition of the regularized Gamma functions 
\begin{eqnarray*}
P \left(\alpha,\beta\right) \defeq \frac{\gamma \left( \alpha , \beta \right)}{\Gamma \left( \alpha \right)}, \quad
Q \left(\alpha,\beta\right) \defeq \frac{\Gamma \left( \alpha , \beta \right)}{\Gamma \left( \alpha \right)},
\end{eqnarray*}
where $\gamma \left(\alpha,\beta\right), \Gamma \left(\alpha,\beta\right)$ and 
$\Gamma \left(\alpha\right)$ are, respectively, the lower incomplete, the upper incomplete and the complete Gamma functions, see~\cite{abramowitz}. 
We also have that 
$\Gamma \left(\alpha\right) = \Gamma \left(\alpha,\beta\right) + \gamma \left(\alpha,\beta\right) $. 
Further, define
\begin{subequations}
\begin{eqnarray}
\Phi_{\theta}(x) \defeq P \left(\frac{x}{2},\frac{\tau (1-\theta) x}{2}\right) + 
Q \left(\frac{x}{2},\frac{\tau (1+\theta) x}{2}\right) ,
\label{gamma_fun}
\end{eqnarray}
where
\begin{eqnarray}
\tau = \frac{\ln(1+\theta)-\ln(1-\theta)}{2 \theta}.
\label{factor}
\end{eqnarray}
\label{12}
\end{subequations}

\begin{theorem}
Let $A$ be a rank-$r$ SPSD $n \times n$  matrix, and let $(\veps,\delta)$ be a tolerance pair. 
If the inequality~\eqref{prob_tr} with $D = G$ holds for some $N$,
then necessarily
\begin{equation}
	\Phi_{\veps}(Nr) \leq \delta.
\label{gauss_bd_03}
\end{equation}
\label{gauss_thm_02}
\end{theorem}

\begin{proof}
As in the proof of Theorem~\ref{gauss_thm_01} we have
\begin{eqnarray*}
Pr\left(| tr_{G}^{N}(A) - tr(A) | \leq \veps~ tr(A) \right) &=& Pr\left(| \sum_{i=1}^{N} \sum_{j=1}^{r} \lambda_{j} z_{ij}^{2}  - N tr(A) | \leq \veps N tr(A)\right) \\
&=& Pr\left((1-\veps) \leq  \sum_{i=1}^{N} \sum_{j=1}^{r} \frac{\lambda_{j}}{tr(A)~N} z_{ij}^{2}  \leq (1+\veps) \right) .
\end{eqnarray*}
Next, 
applying Theorem~3 of~\cite{szba} gives
\begin{equation*}
Pr \left(| tr_{G}^{N}(A) - tr(A) | \leq \veps~ tr(A) \right) \leq 
Pr \left(c (1-\veps) \leq  \frac{1}{Nr} \mathcal{X}_{Nr}^{2}  \leq c(1+\veps) \right) ,
\end{equation*}
where $\mathcal{X}_{M}^{2}$ denotes a chi-squared random variable of degree $M$ with the 
cumulative distribution function
\begin{equation*}
CDF_{\mathcal{X}_{M}^{2}}(x) = Pr \left( \mathcal{X}_{M}^{2} \leq x \right) = \frac{\gamma \left( \frac{M}{2},\frac{x}{2} \right)}{\Gamma \left( \frac{M}{2}\right)} .
\end{equation*}
A further straightforward manipulation yields the stated result.
$\blacksquare$
\end{proof}


Having a computable necessary condition is practically useful: given a pair of fixed sample size $N$ and 
error tolerance $\veps$, 
the failure probability $\delta$ cannot be smaller than $\delta_{0} = \Phi_{\veps}(Nr)$. 

Since our sufficient bounds are not tight, it is not possible to make a direct comparison between the Hutchinson and Gaussian
methods based on them. 
However, using this necessary condition can help for certain matrices.
Consider a low rank matrix with a rather small $\KH$ in~\eqref{hutch_bd_02}.
For such a matrix and a given pair $(\veps,\delta)$, 
the condition~\eqref{gauss_bd_03} will probabilistically necessitate a rather large $N$, 
while~\eqref{hutch_bd_02} may give a much smaller sufficient bound for $N$. 
In this situation, using Theorem~\ref{gauss_thm_02}, the Hutchinson method is indeed guaranteed to require 
a smaller sample size than the Gaussian method.

The condition in Theorem~\ref{gauss_thm_02} is almost tight in the following sense. 
Note that in \eqref{factor}, $\tau \approx 1$ for $\theta = \veps$ sufficiently small.  
So, $1-\Phi_{\veps}(Nr)$ would be very close to 
$Pr\left( (1-\veps) \leq tr_{G}^{N}(A^{*}) \leq (1+\veps) \right)$, 
where $A^{*}$ is an SPD matrix of the same rank as $A$
whose eigenvalues are all equal to $1/r$. 
Next note that the condition~\eqref{gauss_bd_03} should hold for all matrices of the same rank;
hence it is almost tight. Figures~\ref{Nec_N_Rank} and~\ref{fig:thetas} demonstrate this effect. 


Notice that for a very low rank matrix and a reasonable pair $(\veps,\delta)$, 
the necessary $N$ given by~\eqref{gauss_bd_03} could be even larger than the matrix size $n$, 
rendering the Gaussian method useless for such instances; see Figure~\ref{Nec_N_Rank}. 

\begin{figure}[htb]
\centering 
\subfigure[$\veps = \delta = 0.02$]{
\includegraphics[scale = 0.45]{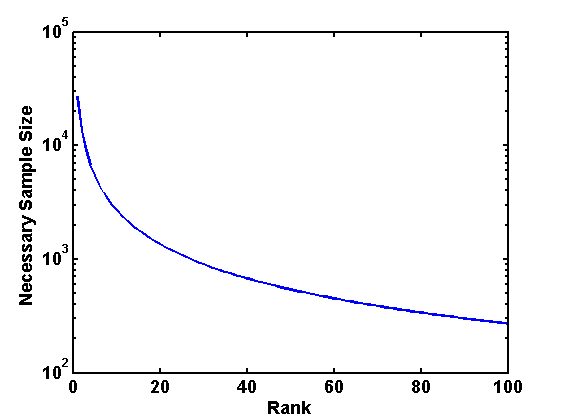}}
\subfigure[$n=10,000, ~ \veps = \delta = 0.1$]{
\includegraphics[scale = 0.45]{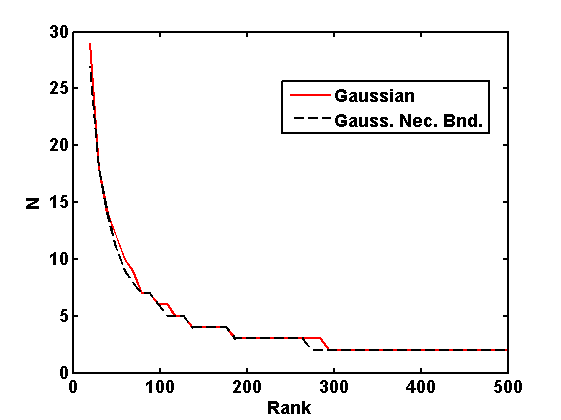}}
\caption{Necessary bound for Gaussian estimator: (a) the log-scale of $N$ according to~\eqref{gauss_bd_03} as a function of $r=rank(A)$: larger ranks yield smaller necessary sample size. 
For very low rank matrices, the necessary bound grows significantly:
for $n=1000$ and $r \leq 30$, necessarily $N > n$ and the Gaussian method is practically useless; 
(b) tightness of the necessary bound demonstrated by an actual run as described for Example~\ref{exp_rank} in Section~\ref{numer} where $A$ has all eigenvalues equal.}
\label{Nec_N_Rank}%
\end{figure}


\section{Random unit vector bounds, with and without replacement, for general square matrices}
\label{randsamp}

An alternative to the Hutchinson and Gaussian estimators is to draw the vectors $\ww_i$ from among the 
$n$ columns of the scaled identity matrix $\sqrt{n} I$.
Note that if $\ww_i$ is the $i$th (scaled) unit vector then $\ww_i^tA\ww_i = n a_{ii}$. Hence the trace can be
recovered in $N=n$ deterministic steps upon setting in \eqref{tr_moncar} $i=j,\ j=1, 2, \ldots , n$.
However, our hope is that for some matrices a good approximation for the trace can be recovered in $N \ll n$
such steps, with $\ww_i$'s drawn as mentioned above.
  
 There are typically two ways one can go about drawing such samples: with or without replacement.
The first of these has been studied in~\cite{avto}. 
However, in view of the exact procedure, we may expect
to occasionally require smaller sample sizes by using the strategy of sampling without replacement.
In this section we make this intuitive observation more rigorous.



In what follows, $U_{1}$ and $U_{2}$ refer to the uniform distribution of
unit vectors with and without replacement, respectively.
We first find expressions for the mean and variance of both strategies, obtaining a smaller variance for $U_2$.

\begin{lemma}
Let $A$ be an $n \times n$ matrix and let $N$ denote the sample size. Then
			\begin{subequations}
			\begin{eqnarray}
					\mathbb{E} \left( tr^{N}_{U_{1}}(A) \right ) &=& \mathbb{E} \left( tr^{N}_{U_{2}}(A) \right ) = tr(A)  \label{mean_R1} , \\
					Var \left( tr^{N}_{U_{1}}(A) \right ) &=& \frac{1}{N} \left( n \sum_{j=1}^{n} a^{2}_{jj}  - tr (A)^2 \right) \label{var_R1} ,\\
					Var \left( tr^{N}_{U_{2}}(A) \right ) &=& \frac{(n-N)}{N (n-1)} \left( n \sum_{j=1}^{n} a^{2}_{jj} - tr (A)^2 \right), \quad N \leq n .\label{var_R2} 
			\end{eqnarray}
			\end{subequations}
\label{randsamp_lem}
\end{lemma}

\begin{proof}
The results for $U_1$ are proved in \cite{avto}. Let us next concentrate on $U_2$, and group the randomly selected unit vectors into an $n \times N$ matrix $W$. Then
\begin{eqnarray*}
\mathbb{E} \left( tr^{N}_{U_{2}}(A) \right ) 
 = \frac{1}{N}  \mathbb{E} \left( tr\left( W^t A W \right) \right) = \frac{1}{N}  \mathbb{E} \left( tr\left( A \; W W^{t} \right) \right) = \frac{1}{N}  tr \left( A \; \mathbb{E} \left( W W^{t} \right) \right) .\nonumber \\
\end{eqnarray*}
Let $y_{ij}$ denote the $(i,j)$th element of
the random matrix $W W^{t}$.
Clearly,  $y_{ij} = 0$ if $i \neq j$.
It is also easily seen that $y_{ii}$ can only take on the values $0$ or $n$. We have
\begin{eqnarray*}
\mathbb{E} \left( y_{ii}  \right) = n Pr\left(y_{ii} = n \right) = n \frac{\binom{n-1}{N-1}}{\binom{n}{N}} = N,
\end{eqnarray*}
so $\mathbb{E} ( W W^{t} ) = N \cdot I$, where $I$ stands for the identity matrix. 
This, in turn, gives $\mathbb{E} \left( tr^{N}_{U_{2}}(A) \right ) = tr(A)$.

For the variance, we first calculate
\begin{eqnarray}
&& \mathbb{E} \left [ \left( tr^{N}_{U_{2}}(A) \right )^2 \right ] = \frac{1}{N^2} \mathbb{E} \left( \sum_{i=1}^{N} \sum_{j=1}^{N} \left( \ww^{t}_{i} A \ww_{i} \right) \left(\ww^{t}_{j} A \ww_{j} \right) \right) \nonumber \\
&=& \frac{1}{N^{2}} \left(\sum_{i=1}^{N} \mathbb{E} \left[ \left( \ww^{t}_{i} A \ww_{i} \right)^{2} \right] + \sum_{i=1}^{N} \sum_{\substack{j=1 \\ j\neq i}}^{N} \mathbb{E} \left[ \left( \ww^{t}_{i} A \ww_{i} \right) \left( \ww^{t}_{j} A \ww_{j} \right) \right]\right) .
\label{tr_U_2_temp}
\end{eqnarray}

Let $\ee_{j}$ denote the $j^{th}$ column of the scaled identity matrix, $\sqrt{n} I$. 
Using the law of total expectation (i.e., the tower rule),
we have for any two random vectors $\ww_{i}$ and $\ww_{j}$ with $i \neq j$,
\begin{eqnarray*}
\mathbb{E} \left[ \left( \ww^{t}_{i} A \ww_{i} \right) \left( \ww^{t}_{j} A \ww_{j} \right) \right] &=& \sum_{k=1}^{n} \mathbb{E} \left[ \left( \ww^{t}_{i} A \ww_{i} \right) \left( \ww^{t}_{j} A \ww_{j} \right)  | \ww_{i} = \ee_{k} \right] \cdot Pr \left( \ww_{i} = \ee_{k} \right) \nonumber \\
&=& \sum_{k=1}^{n} n a_{kk} \cdot \mathbb{E} \left[ \left( \ww^{t}_{j} A \ww_{j} \right)  | \ww_{i} = \ee_{k} \right] \cdot \frac{1}{n} \nonumber \\
&=& \sum_{k=1}^{n} a_{kk} \sum_{\substack{l=1 \\ l\neq k}}^{n} \mathbb{E} \left[ \left( \ww^{t}_{j} A \ww_{j} \right)  | \ww_{j} = \ee_{l} \right] \cdot Pr \left( \ww_{j} = \ee_{l} \right | \ww_{i} = \ee_{k}) \nonumber \\
&=& \sum_{k=1}^{n} a_{kk} \sum_{\substack{l=1 \\ l\neq k}}^{n} n a_{ll} \frac{1}{n-1} 
= \frac{n}{n-1}  \sum_{k=1}^{n} \sum_{\substack{l=1 \\ k\neq l}}^{n} a_{kk} a_{ll} \\
&=& \frac{n}{n-1}   ( tr(A)^{2} - \sum_{j=1}^{n} a^{2}_{jj} ).
\end{eqnarray*}
Substituting this in~\eqref{tr_U_2_temp} gives 
\begin{eqnarray*}
\mathbb{E} \left [ \left( tr^{N}_{U_{2}}(A) \right )^2 \right ] &=& \frac{1}{N^{2}} \left( n N \sum_{j=1}^{n} a^{2}_{jj} + \frac{n N (N-1)}{n-1} ( tr(A)^{2} - \sum_{j=1}^{n} a^{2}_{jj} ) \right) .
\end{eqnarray*}
Next, the variance is 
\begin{equation*}
Var \left( tr^{N}_{U_{2}}(A) \right ) = \mathbb{E} \left [ \left( tr^{N}_{U_{2}}(A) \right )^2 \right ] - \left [\mathbb{E} \left( tr^{N}_{U_{2}}(A) \right ) \right ]^{2} ,
\end{equation*}
which gives \eqref{var_R2}.
$\blacksquare$
\end{proof}

Note that $Var \left( tr^{N}_{U_{2}}(A) \right ) = \frac{n-N}{n-1} Var \left( tr^{N}_{U_{1}}(A) \right )$.
The difference in variance between these sampling strategies is small for $N \ll n$, and they coincide if
$N=1$. 
Moreover, in case that the diagonal entries of the matrix are all equal, the variance for both sampling strategies vanishes.

We now turn to the analysis of the sample size required to ensure~\eqref{prob_tr} and find a slight improvement over 
the bound given in~\cite{avto}  for $U_1$. 
A similar analysis for the case of sampling without replacement shows that the latter
may generally be a somewhat better strategy.

\begin{theorem}
Let $A$ be a real $n \times n$ matrix, and denote
\begin{equation}
\KU^{(i,j)} =  \frac{n}{tr(A)} \left| a_{ii} - a_{jj} \right| ,\quad \KU = \max_{\substack{1 \leq i,j \leq n \\ i \neq j}} \KU^{(i,j)}.
\label{randsamp1_energy_dist}
\end{equation}
Given a pair of positive small values
$(\veps,\delta)$, the inequality~\eqref{prob_tr} holds
with $D=U_{1}$ if 
\begin{equation}
N > \frac{\KU^{2}}{2} c(\veps,\delta) \equiv \mathcal{F} ,
\label{randsamp1_bd}
\end{equation}
and with $D=U_{2}$ if
\begin{equation}
N \geq \frac{n+1}{1+\frac{n-1}{\mathcal{F}}} .
\label{randsamp2_bd}
\end{equation}
\label{randsamp_thm}
\end{theorem}

\begin{proof}
This proof is refreshingly short. Note first that
every sample of these estimators takes on a Rayleigh value in $[n \min_{j} a_{jj} , n \max_{j} a_{jj}]$.

The proof of~\eqref{randsamp1_bd}, for the case with replacement,
uses Hoeffding's inequality in exactly the same way as the corresponding theorem in~\cite{avto}.
We obtain directly that if~\eqref{randsamp1_bd} is satisfied then \eqref{prob_tr} holds with $D = U_1$.

For the case without replacement we use Serfling's inequality~\cite{serf} to obtain
\begin{equation*}
Pr\left(| tr_{U_{2}}^{N}(A) - tr(A) | \geq \veps tr(A) \right ) \leq 2 \exp \left\{ \frac{-2 N \veps^2 }{\left(1-f_{N-1}\right) \KU^{2}}  \right\} ,
\end{equation*}
where $f_N$ is the sampling fraction defined as
\begin{equation*}
f_{N} = \frac{N-1}{n-1} .
\end{equation*}
Now, for the inequality~\eqref{prob_tr} to hold, we need
\begin{equation*}
2 \exp \left\{ \frac{-2 N \veps^2 }{\left(1-f_{N-1}\right) \KU^{2}} \right\} \leq \delta ,
\end{equation*}
so we require that
\begin{equation*}
\frac{N}{1-f_{N-1}} \geq \mathcal{F} .
\end{equation*}
The stated result \eqref{randsamp2_bd} is obtained following 
some straightforward algebraic manipulation.
$\blacksquare$
\end{proof}
 
Looking at the bounds~\eqref{randsamp1_bd} for $U_1$ and~\eqref{randsamp2_bd} for $U_2$ and observing the expression \eqref{randsamp1_energy_dist} for $\KU$, one can gain insight as to the type of matrices which are handled efficiently using this estimator:
this would be the case
if the diagonal elements of the matrix all have similar values.  
In the extreme case where they are all the same, we only need one sample.
The corresponding expression in \cite{avto} does not reflect this result.
An illustration of the relative behaviour of the two bounds is given in Figure~\ref{N_randsamp_comp}.

\begin{figure}[htb]
\centering 
\mbox{
\includegraphics[scale = 0.5]{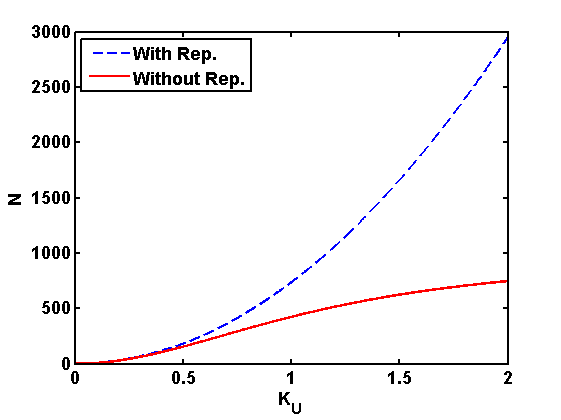}
}
\caption{The behaviour of the bounds~\eqref{randsamp1_bd} and~\eqref{randsamp2_bd} with respect to the factor 
$K = \KU$ for $n=1000$ and $\veps = \delta = 0.05$. 
The bound for $U_2$ is much more resilient to the distribution of the diagonal values than that of $U_1$. 
For very small values of $\KU$, there is no major difference between the bounds.}
\label{N_randsamp_comp}%
\end{figure}


\section{Numerical Examples}
\label{numer}

In this section we experiment with several examples, comparing the performance of different methods with regards to various matrix properties and verifying that the bounds obtained in our theorems indeed agree with the numerical experiments. 


\begin{example}
\label{examp0}
\begin{figure}[htb]
\centering 
\mbox{
\includegraphics[scale = 0.45]{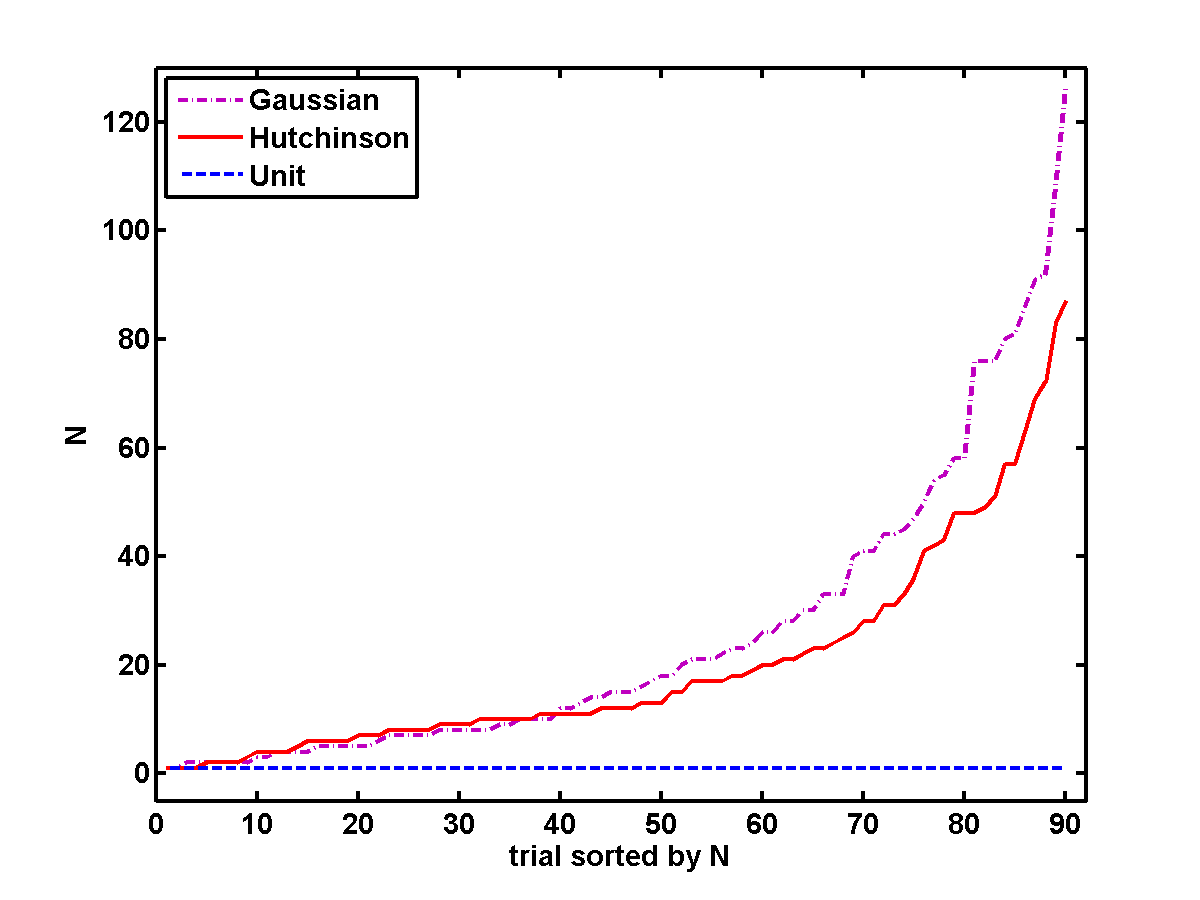}
}
\caption{Example~\ref{examp0}. For the matrix of all $1$s with $n=10,000$, the plot
depicts the numbers of samples in 100 trials required to satisfy the relative tolerance $\veps = .05$,
sorted by increasing $N$. The average $N$ for both Hutchinson and Gauss estimators was around $50$,
while for the uniform unit vector estimator always $N=1$.
Only the best 90 results (i.e., lowest resulting values of $N$)
are shown for reasons of scaling.
Clearly, the unit vector method is superior here.}
\label{fig:all1s}%
\end{figure}
In this example we do not consider $\delta$ at all.
Rather, we check numerically for various values of $\veps$ what value of $N$ is required to achieve
a result respecting this relative tolerance. We have calculated maximum and average values for $N$
over 100 trials for several special examples, verifying numerically the following considerations. 
\begin{itemize}
\item
The matrix of all $1$s (in {\sc Matlab}, {\tt A=ones(n,n)}) has been considered
in~\cite{avto}. Here $tr(A) = n, \ \KH = n-1$, and
a very large $N$ is often required if $\veps$ is small for both Hutchinson and Gauss methods.
For the unit vector method, however, $\KU = 0$ in \eqref{randsamp1_energy_dist}, 
so the latter method converges in one iteration, $N = 1$. This fact
yields an example where the unit vector estimator is
far better than either Hutchinson or Gaussian estimators; see Figure~\ref{fig:all1s}.
\item
Another extreme example, where this time it is the Hutchinson estimator which requires only one sample
whereas the other methods may require many more, is the case of a diagonal matrix $A$. For a diagonal
matrix, $\KH = 0$, and the result follows from Theorem~\ref{hutch_thm_02}.
\item
If $A$ is a multiple of the identity then, since $\KU = \KH = 0$,
only the Gaussian estimator from among the methods considered requires more than one sample;
thus, it is worst. 
\item
Examples where the unit vector estimator is consistently (and significantly)
worst are obtained by defining $A = Q^tDQ$ for a diagonal matrix $D$ with different
positive elements which are of the same order of magnitude and a nontrivial orthogonal matrix $Q$.
\item
We have not been able to come up with a simple example of the above sort where the Gaussian estimator
shines over both others, although we have seen many occasions in practice where it slightly outperforms the Hutchinson estimator
with both being significantly better than the unit vector estimators.
\end{itemize}
\end{example}


\begin{example}
\label{examp1}

Consider the matrix $A = \xx \xx^{t}/\|\xx\|^{2}$, where $\xx \in \mathbb{R}^{n}$, 
and for some $\theta > 0$, $x_{j} = \exp(-j \theta ), \ 1 \leq j \leq n$.
This extends the example of all 1s of Figure~\ref{fig:all1s} (for which $\theta = 0$)
to instances with rapidly decaying elements.

It is easy to verify that 
\begin{eqnarray*}
& & tr(A) = 1, \quad r = 1, \quad \KG = 1, \\ 
& & \KH^{j} = \| \xx \|^{2} x_{j}^{-2} - 1, \quad \KH = \| \xx \|^{2} x_{n}^{-2} - 1, \\ 
& & \KU^{(i,j)} = \frac{n}{\| \xx \|^{2}} |x_{i}^{2} - x_{j}^{2}| , \quad  
\KU = \frac{n}{\| \xx \|^{2}} ( x_1^2 - x_n^2), \\
& & \| \xx \|^{2} = \frac{\exp(-2 \theta) - \exp(-2 (n+1) \theta)}{1- \exp(-2 \theta)}.
\end{eqnarray*}

\begin{figure}[htb]
\centering 
\mbox{
\includegraphics[scale = 0.45]{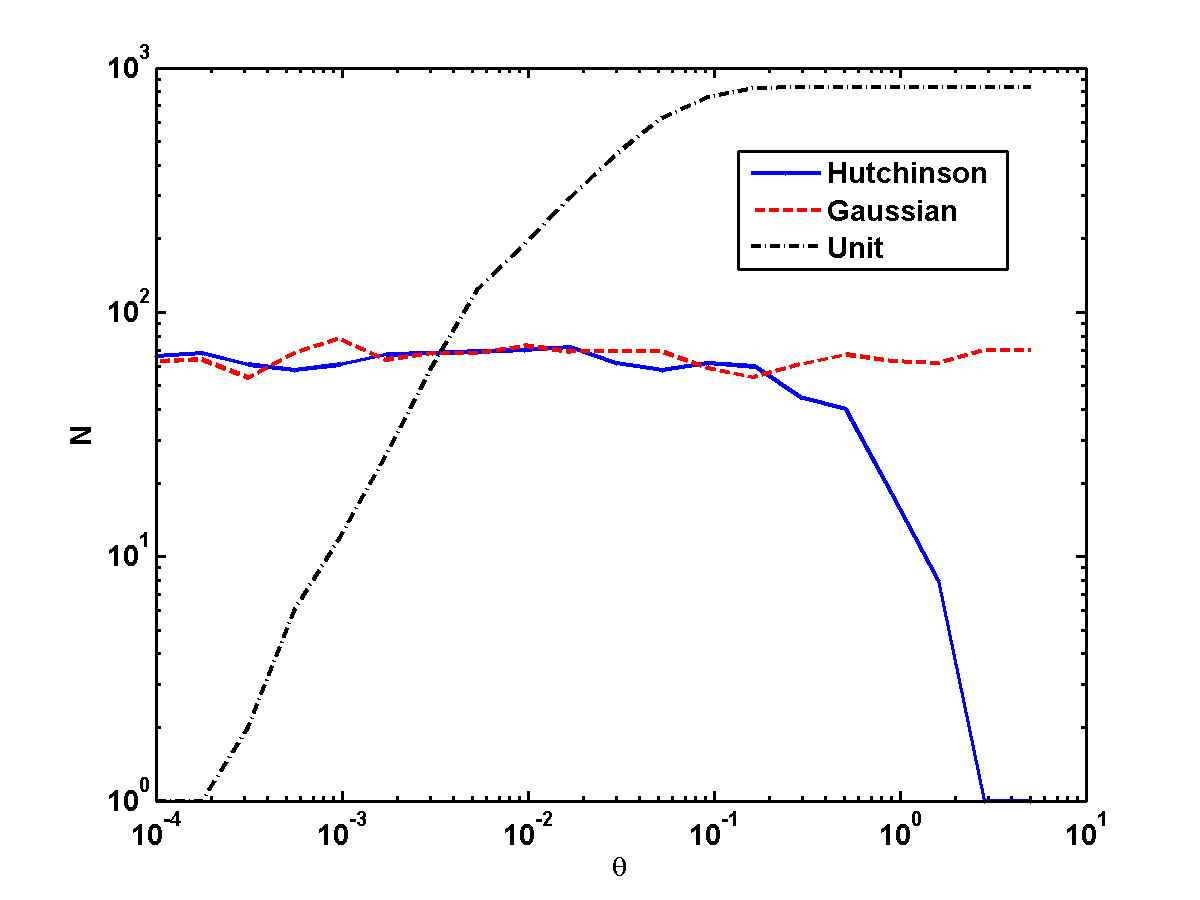}
}
\caption{Example~\ref{examp1}. For the rank-1 matrix arising from a rapidly-decaying vector with $n=1000$, 
this log-log plot depicts the actual sample size $N$ required for~\eqref{prob_tr} to hold with $\veps = \delta = 0.2$,  
vs. various values of $\theta$. In the legend, ``Unit'' refers to the random sampling method without replacement.}
\label{fig:thetas}%
\end{figure}

Figure~\ref{fig:thetas} displays the ``actual sample size'' $N$ for a particular pair $(\veps, \delta )$
as a function of $\theta$ for the three distributions.
The values $N$ were obtained by running the code 100 times for each $\theta$
to calculate the empirical probability of success.

In this example the distribution of $\KH^{j}$ values gets progressively worse with heavier tail 
values as $\theta$ gets larger. However, recall that this matters in terms of the sufficient bounds
\eqref{hutch_bd_01} and \eqref{hutch_bd_02}
only so long as $\KH < 3$. Here the crossover point happens roughly when $\theta \sim 1/(2n)$.
Indeed, for large values of $\theta$ the required sample size actually drops when using the Hutchinson method:
Theorem~\ref{hutch_thm_02}, being only a sufficient condition, merely distinguishes 
types of matrices for which Hutchinson is expected to be {\em efficient}, 
while making no claim regarding those matrices for which it is an inefficient estimator.

On the other hand, Theorem~\ref{gauss_thm_02}
clearly distinguishes the types of matrices for which the Gaussian method
is expected to be {\em inefficient}, because its condition is necessary
rather than sufficient. Note that $N$ (the red curve in Figure~\ref{fig:thetas})
does not change much as a function of $\theta$,
which agrees with the fact that the matrix rank stays fixed and low at $r=1$.

The unit vector estimator, unlike Hutchinson, deteriorates steadily as $\theta$ is increased,
because this estimator ignores off-diagonal elements. However,
for small enough values of $\theta$ the $\KU^{(i,j)}$'s are spread tightly near zero, 
and the unit vector method, as predicted by Theorem~\ref{randsamp_thm}, requires a very small sample size.

\end{example}

For Examples~\ref{exp_practical} and~\ref{exp_properties} below, given $(\veps,\delta)$, we plot the probability of success, i.e., 
$Pr\left(| tr_{D}^{N}(A) - tr(A) | \leq \veps~ tr(A) \right)$ for increasing values of $N$, starting from $N=1$. 
We stop when for a given $N$, the probability of success is greater than or equal to $1-\delta$. 
In order to evaluate this for each $N$,
we run the experiments 500 times and calculate the empirical probability. 

In the figures below, `With Rep.' and `Without Rep.' refer to uniform unit sampling with and without replacement, respectively.
In all cases, by default, $\veps = \delta = .05$.
We also provide distribution plots of the quantities $\KH^{j}, \ \KG^{j}$
and $\KU^{(i,j)}$ appearing in \eqref{hutch_energy_dist}, \eqref{gauss_energy_dist} 
and \eqref{randsamp1_energy_dist}, respectively. These quantities are indicators for the
performance of the Hutchinson, Gaussian and unit vector estimators, respectively, as evidenced not only
by Theorems~\ref{hutch_thm_02}, \ref{gauss_thm_01} and~\ref{randsamp_thm}, but also
in Examples~\ref{examp0} and~\ref{examp1},
and by the fact that the performance of the Gaussian and unit vector estimators
is not affected by the energy of the off-diagonal matrix elements. 


\begin{example}[Data fitting with many experiments]
\label{exp_practical}

A major source of applications where trace estimation is central arises in problems
involving least squares data fitting with many experiments.
In its simplest, linear form, we look for $\mm \in \R^l$ so that the misfit function
\begin{subequations}
\begin{eqnarray}
\phi (\mm) =  \sum_{i=1}^n \| J_i\mm - \dd_i \|^2,
\label{misfita}
\end{eqnarray}
for given data sets $\dd_i$ and sensitivity matrices $J_i$,
is either minimized or reduced below some tolerance level. The $m \times l$ matrices $J_i$
are very expensive to calculate and store, so this is avoided altogether,
but evaluating $J_i\mm$ for any suitable vector
$\mm$ is manageable. Moreover, $n$ is large.
Next, writing \eqref{misfita} using the Frobenius norm as
\begin{eqnarray}
\phi (\mm) =  \| C \|_F^2,
\label{misfitb}
\end{eqnarray}
where $C$ is $m \times n$ with the $j$th column $C_j =  J_j\mm - \dd_j$, and defining the
SPSD matrix $A = C^tC$, we have
\begin{eqnarray}
\phi (\mm) =  tr(A).
\label{misfitc}
\end{eqnarray}
Cheap estimates of the misfit function $\phi(\mm)$ are then sought by approximating the trace in \eqref{misfitc}
using only $N$ (rather than $n$) linear combinations of the columns of $C$,
which naturally leads to expressions of the form \eqref{tr_moncar}.
Hutchinson and Gaussian estimators in a similar or more complex context
were considered in~\cite{HaberChungHermann2010,learhe,yori}.
\label{misfit}
\end{subequations}

Drawing the $\ww_i$ as random unit vectors
instead is a method proposed in~\cite{doas3} and compared to others in~\cite{rodoas1}, where it
is called ``random subset'': this latter method can have efficiency advantages that are beyond the scope of the presentation here. 
%
Typically, $m \ll n$, and thus the matrix $A$ is dense and often has low rank. 

Furthermore, the signs of the entries in $C$ can be, at least to some extent, considered random.
Hence we consider below matrices $A = C^tC$ whose entries are Gaussian random variables,
obtained using the {\sc Matlab} command {\tt C = randn(m,n)}. 
We use $m = 200$ and hence the rank is, almost surely, $r=200$.

\begin{figure}[htb]
\centering
\subfigure[Convergence Rate]{
\includegraphics[scale=0.36]{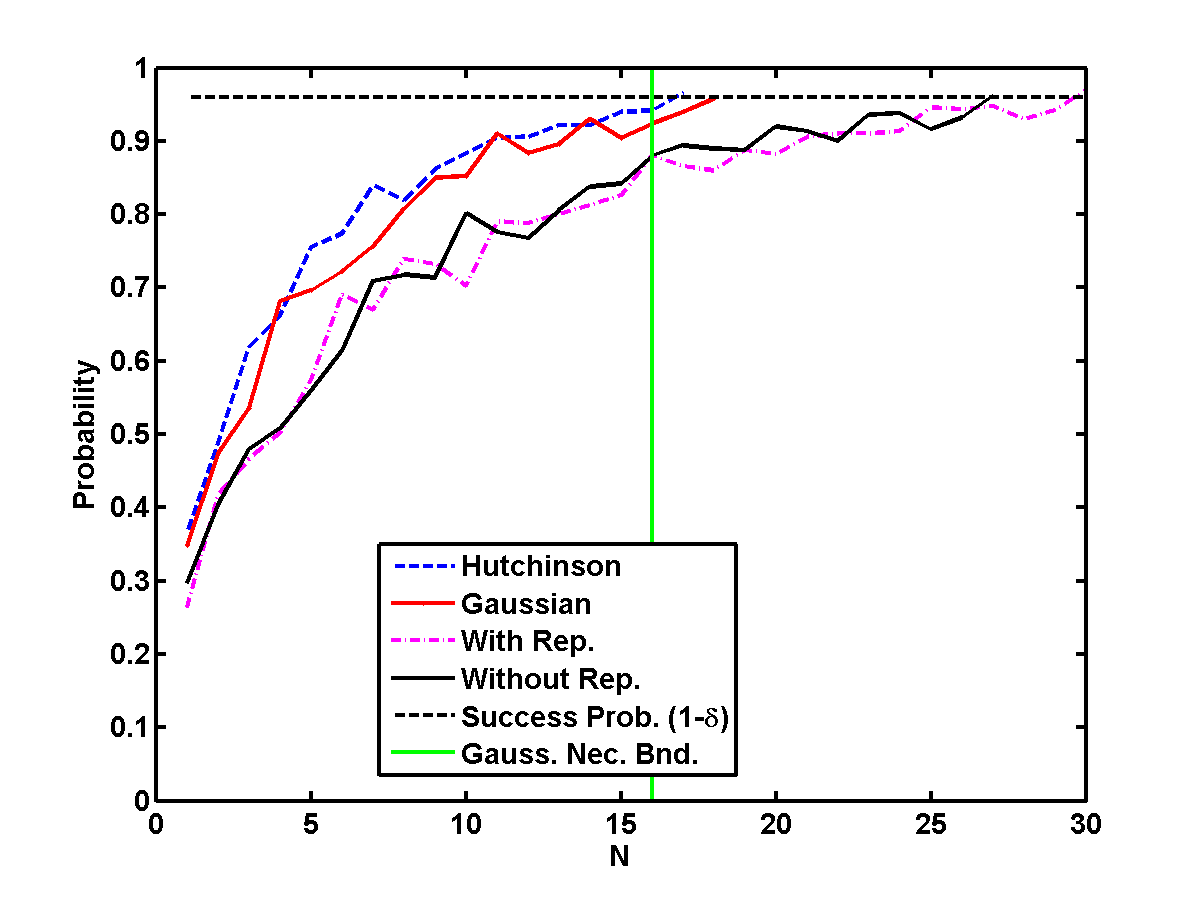}}
\subfigure[$\KH^{j}$ distribution]{
\includegraphics[scale=0.36]{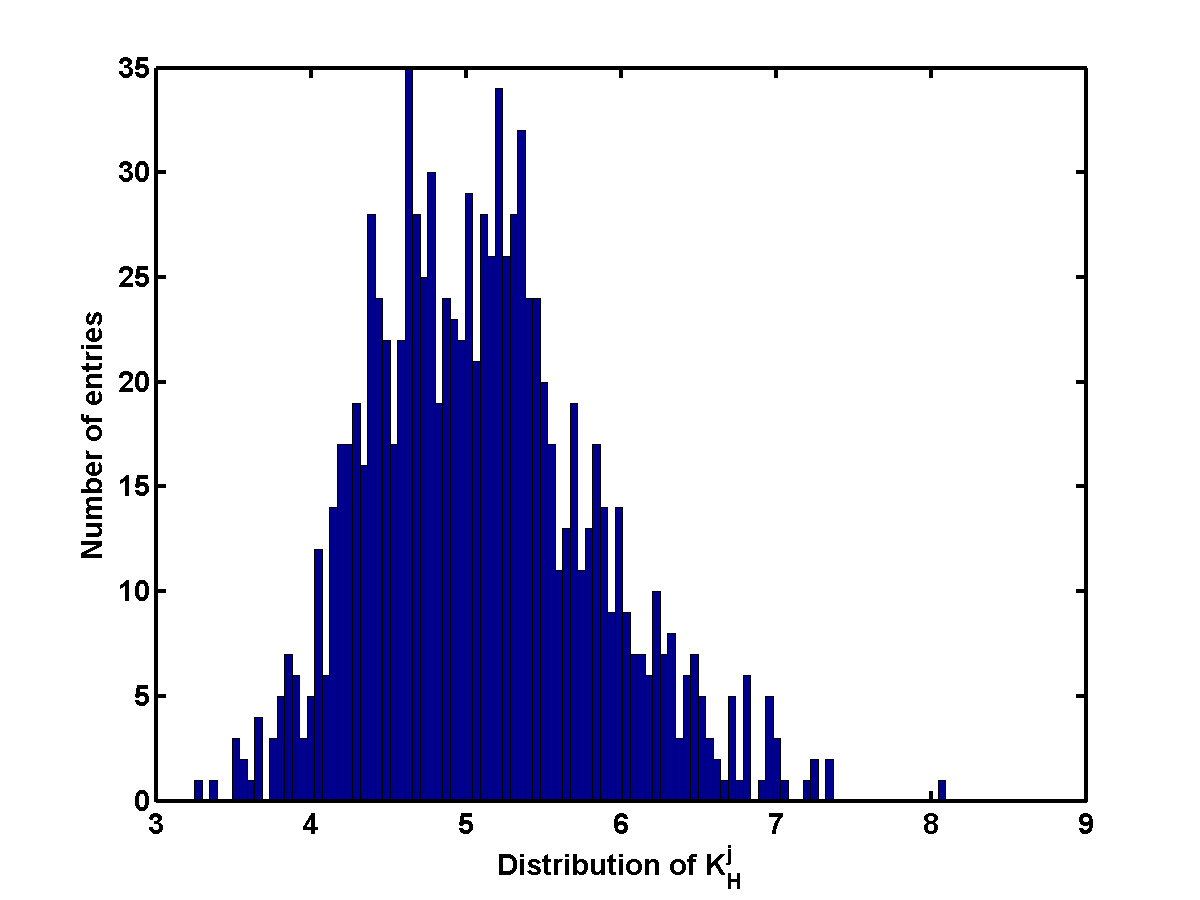}}
\subfigure[$\KU^{(i,j)}$ distribution]{
\includegraphics[scale=0.36]{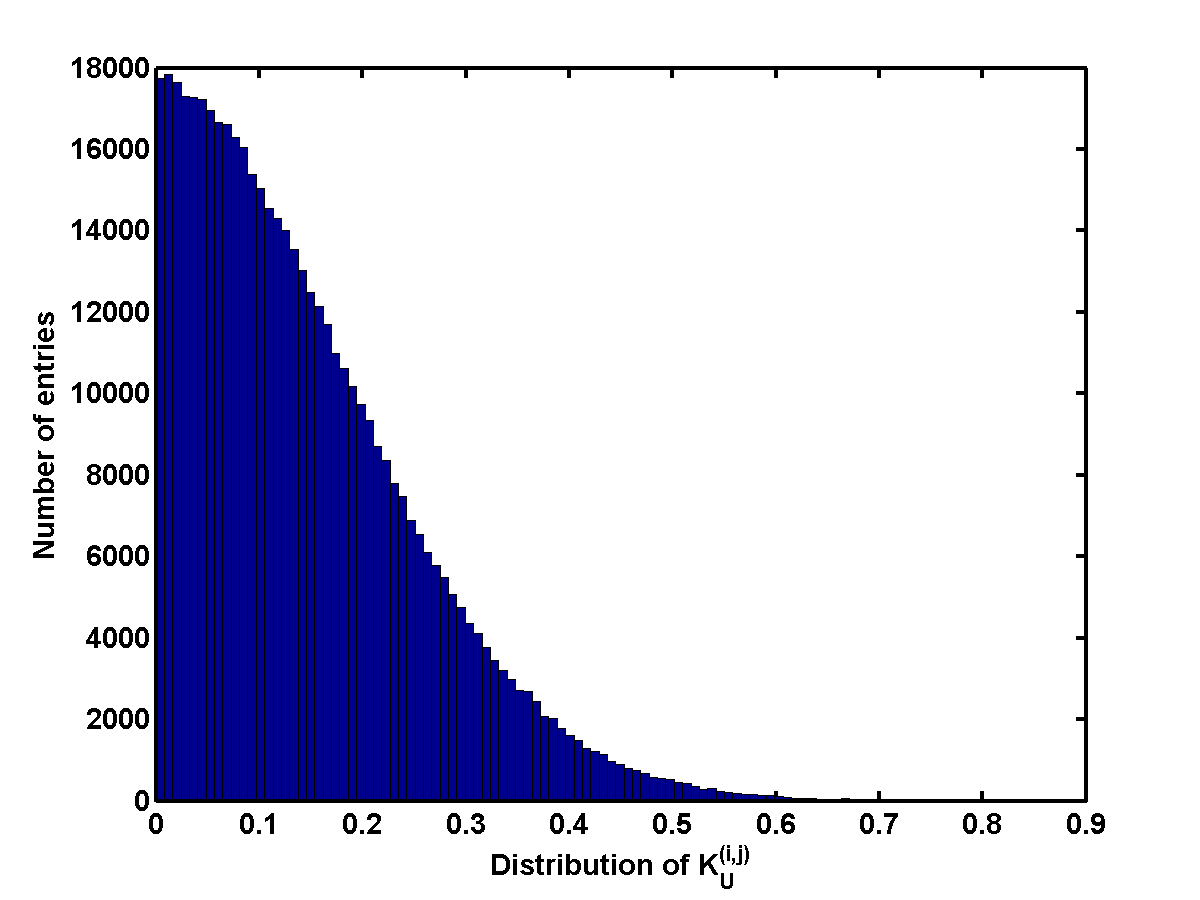}}
\subfigure[Eigenvalue distribution]{
\includegraphics[scale=0.36]{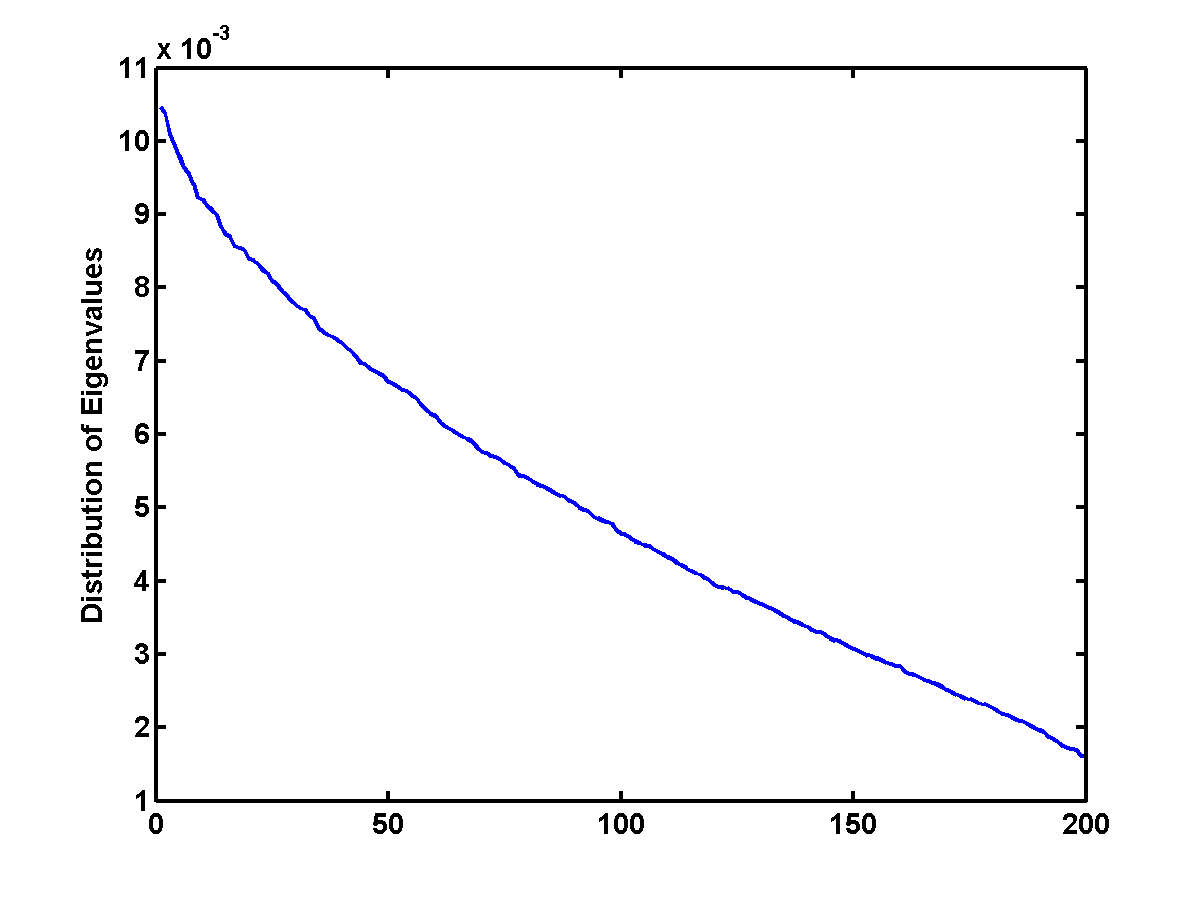}}
\caption{Example~\ref{exp_practical}. A dense SPSD matrix $A$ is constructed using {\sc Matlab}'s {\tt randn}. 
Here $n=1000, \, r=200, \, tr(A) = 1, \KG = 0.0105,\, \KH = 8.4669$ and $\KU = 0.8553$.
The method convergence plots in (a) are for $\veps = \delta = .05$.}
\label{exp_app_fig}
\end{figure}

It can be seen from Figure~\ref{exp_app_fig}(a) that the Hutchinson and the Gaussian methods perform similarly here. 
The sample size required by both unit vector estimators is approximately twice that of the Gaussian and Hutchinson methods. 
This relative behaviour agrees with our observations in the context of actual application as described above, see~\cite{rodoas1}. 
From  Figure~\ref{exp_app_fig}(d), the eigenvalue distribution of the matrix is not very badly skewed, 
which helps the Gaussian method perform relatively well for this sort of matrix. 
On the other hand, by Figure~\ref{exp_app_fig}(b) the relative $\ell_{2}$ energies of the off-diagonals are far from being small,
which is not favourable for the Hutchinson method. These two properties, in combination, result in the similar performance of the Hutchinson and Gaussian
methods despite the relatively low rank. 
The contrast between $\KU^{(i,j)}$'s is not too large according to Figure~\ref{exp_app_fig}(c), 
hence a relatively decent performance of both unit vector (or, random sampling) methods is observed. There is no reason to insist on
avoiding repetition here either.
\end{example}


\begin{example}[Effect of rank and $\KG$ on the Gaussian estimator]
\label{exp_rank}

In this example we plot the actual sample size $N$ required for~\eqref{prob_tr} to hold. 
In order to evaluate~\eqref{prob_tr}, we repeat the experiments 500 times and calculate the empirical probability. 
In all experiments, the sample sizes predicted by~\eqref{hutch_bd_01} and~\eqref{gauss_bd_01} were so pessimistic compared with the true $N$ 
that we simply did not include them in the plots.

\begin{figure}[htb]
\centering 
\subfigure[{\tt sprandn}, $n=5,000$]{
\includegraphics[scale = 0.45]{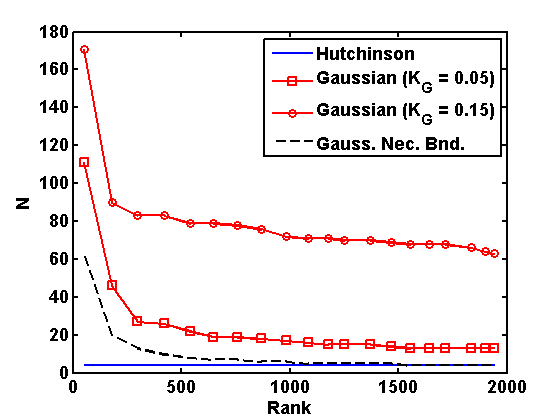}}
\subfigure[diagonal, $n=10,000$]{
\includegraphics[scale = 0.45]{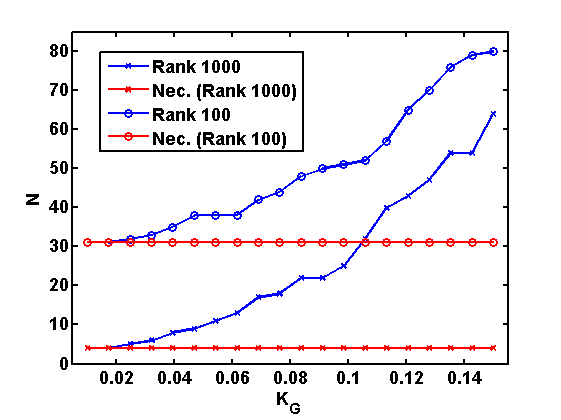}}
\caption{Example~\ref{exp_rank}. The behaviour of the Gaussian method with respect to rank and $\KG$. We set $\veps = \delta = .05$
and display the necessary condition~\eqref{gauss_bd_03} as well.}
\label{Rank_N_G12_H}%
\end{figure}

In order to concentrate only on rank and $\KG$ variation, we make sure that in all experiments $\KH \ll 1$.
For the results displayed in Figure~\ref{Rank_N_G12_H}(a), where $r$ is varied for each of two values of $\KG$, 
this is achieved by playing with {\sc Matlab}'s 
normal random generator function {\tt sprandn}. 
For Figure~\ref{Rank_N_G12_H}(b), where $\KG$ is varied for each of two values of $r$, diagonal matrices are utilized:
we start with a uniform distribution of the eigenvalues and gradually make this distribution more skewed, resulting in an increased $\KG$.
The low $\KH$ values cause the Hutchinson method to look very good, but that is not our focus here.

It can be clearly seen from Figure~\ref{Rank_N_G12_H}(a) that as the matrix rank gets lower, 
the sample size required for the Gaussian method grows significantly. For a given rank, the matrix with a smaller $\KG$ requires smaller sample size.
From Figure~\ref{Rank_N_G12_H}(b) it can also be seen that 
for a fixed rank, the matrix with more skewed $\KG^{j}$'s distribution 
(marked here by a larger $\KG$) requires a larger sample size.

\end{example}
\begin{figure}[htb]
\centering
\subfigure[Convergence Rate]{
\includegraphics[scale=0.36]{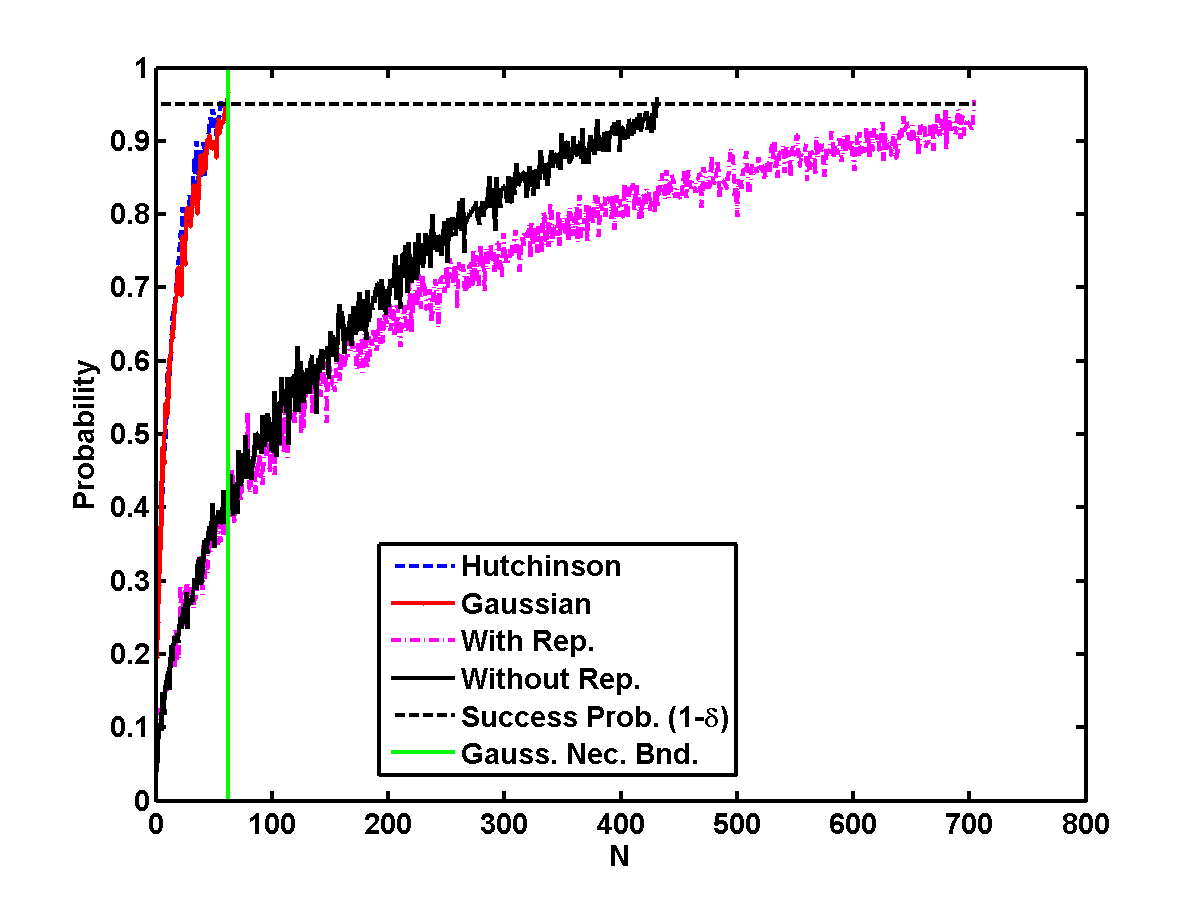}}
\subfigure[$\KH^{j}$ distribution ($\KH^{j} \leq 100$)]{
\includegraphics[scale=0.36]{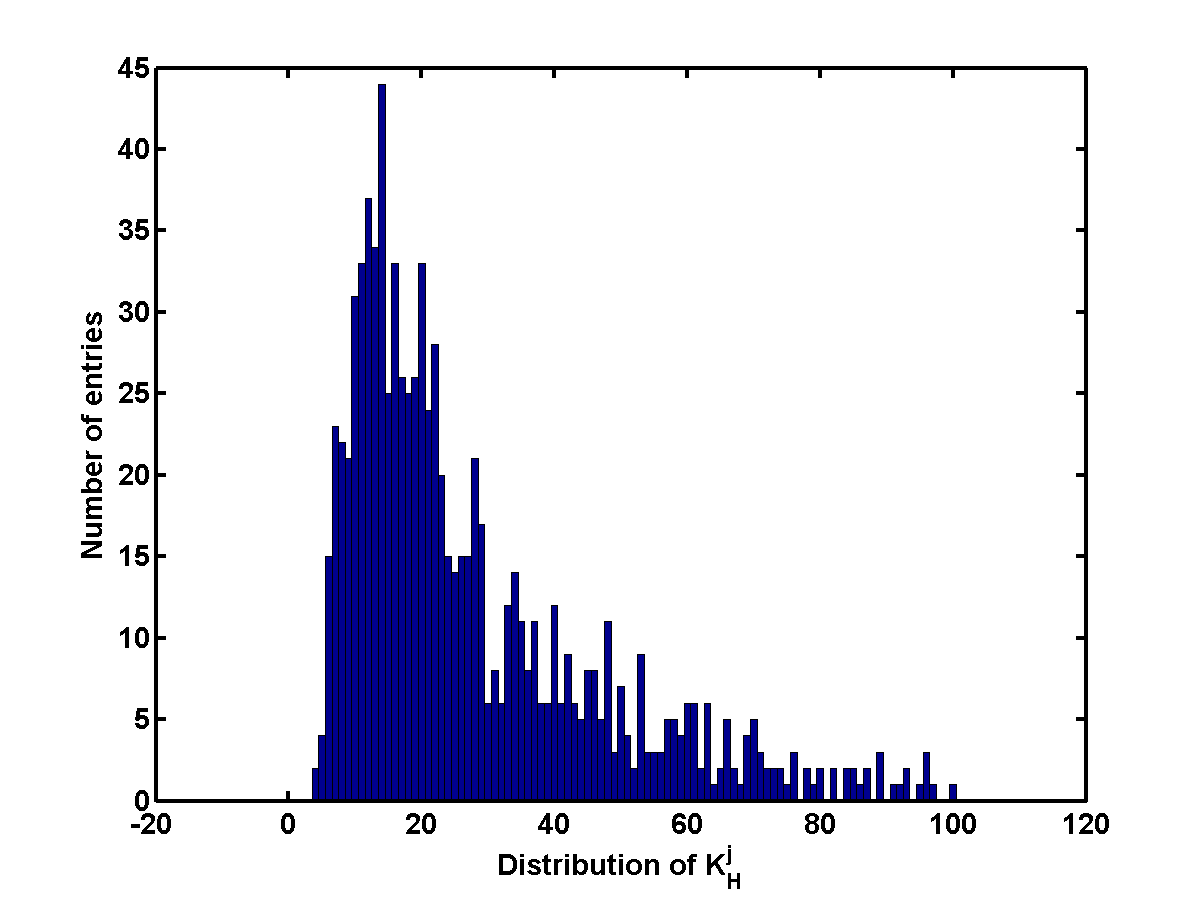}}
\subfigure[$\KU^{(i,j)}$ distribution]{
\includegraphics[scale=0.36]{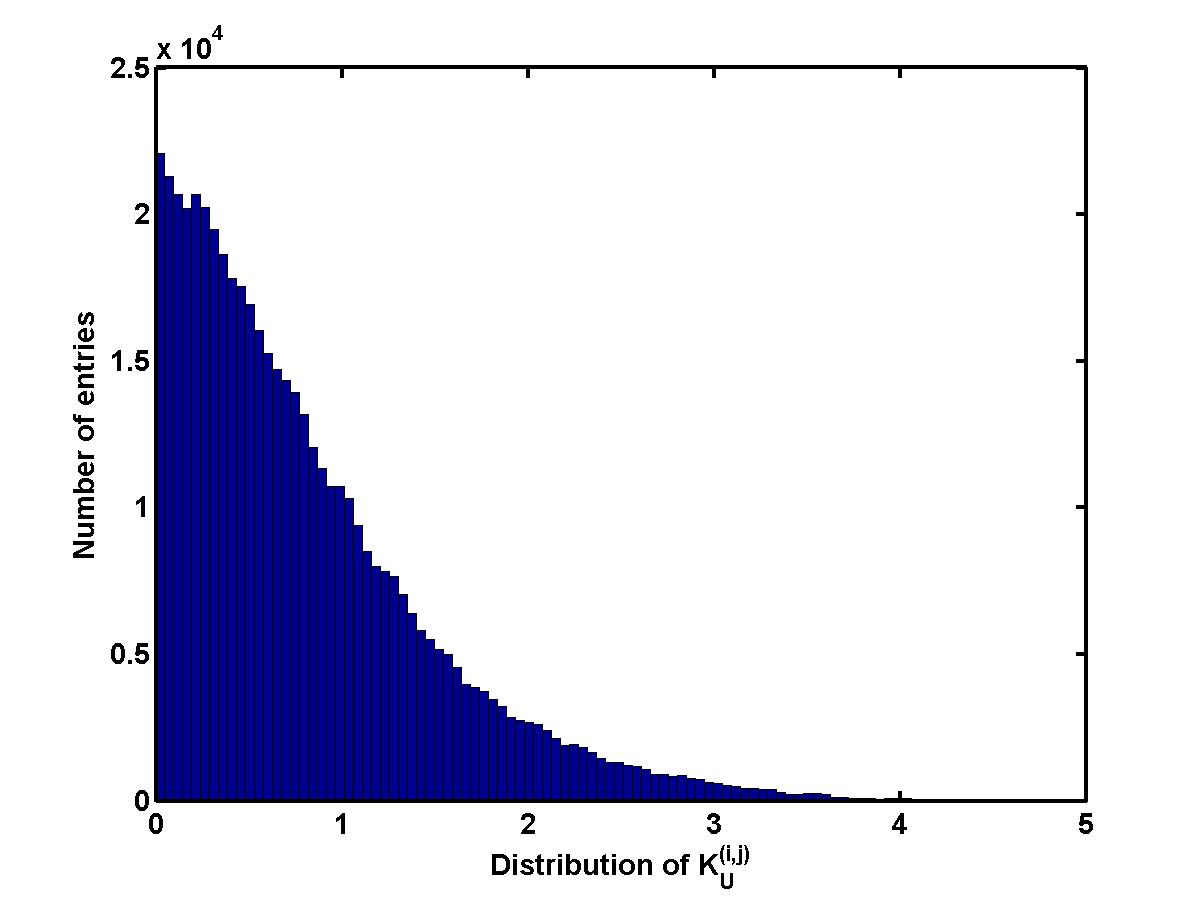}}
\subfigure[Eigenvalue distribution]{
\includegraphics[scale=0.36]{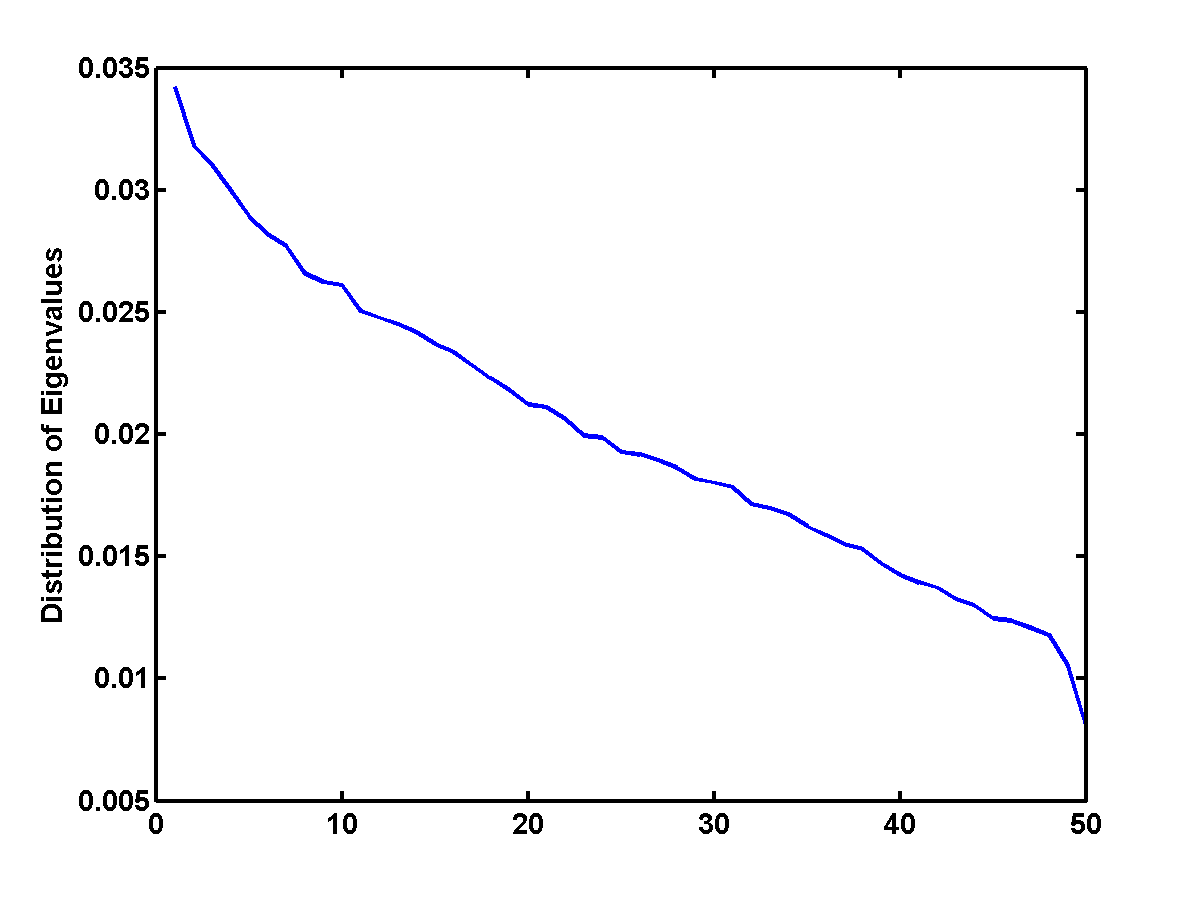}}
\caption{Example~\ref{exp_properties}. A sparse matrix {\tt (d = 0.1)} is formed using {\tt sprandn}. 
Here $r=50, \, \KG = 0.0342, \, \KH = 15977.194$ and $\KU = 4.8350$.}
\label{exp_prop_fig03}
\end{figure}


\begin{example}[Method performance for different matrix properties]
\label{exp_properties}

Next we consider a much more general setting than that in Example~\ref{exp_rank},
and compare the performance of different methods with respect to various matrix properties.
The matrix $A$ is constructed as in Example~\ref{exp_practical}, except that also a uniform distribution is used.
Furthermore, a parameter $d$ controlling denseness of the created matrix is utilized.
This is achieved in {\sc Matlab} using the commands 
{\tt C=sprandn(m,n,d)} or {\tt C=sprand(m,n,d)}.
By changing {\tt m} and {\tt d} we can change the matrix properties $\KH$, $\KG$ and $\KU$ 
while keeping the rank $r$ fixed across experiments.
We maintain $n=1000, \, tr(A) = 1$ and $\veps = \delta = .05$ throughout.
In particular, the four figures related to this example are comparable to Figure~\ref{exp_app_fig} but for a lower rank.

\begin{figure}[htb]
\centering
\subfigure[Convergence Rate]{
\includegraphics[scale=0.36]{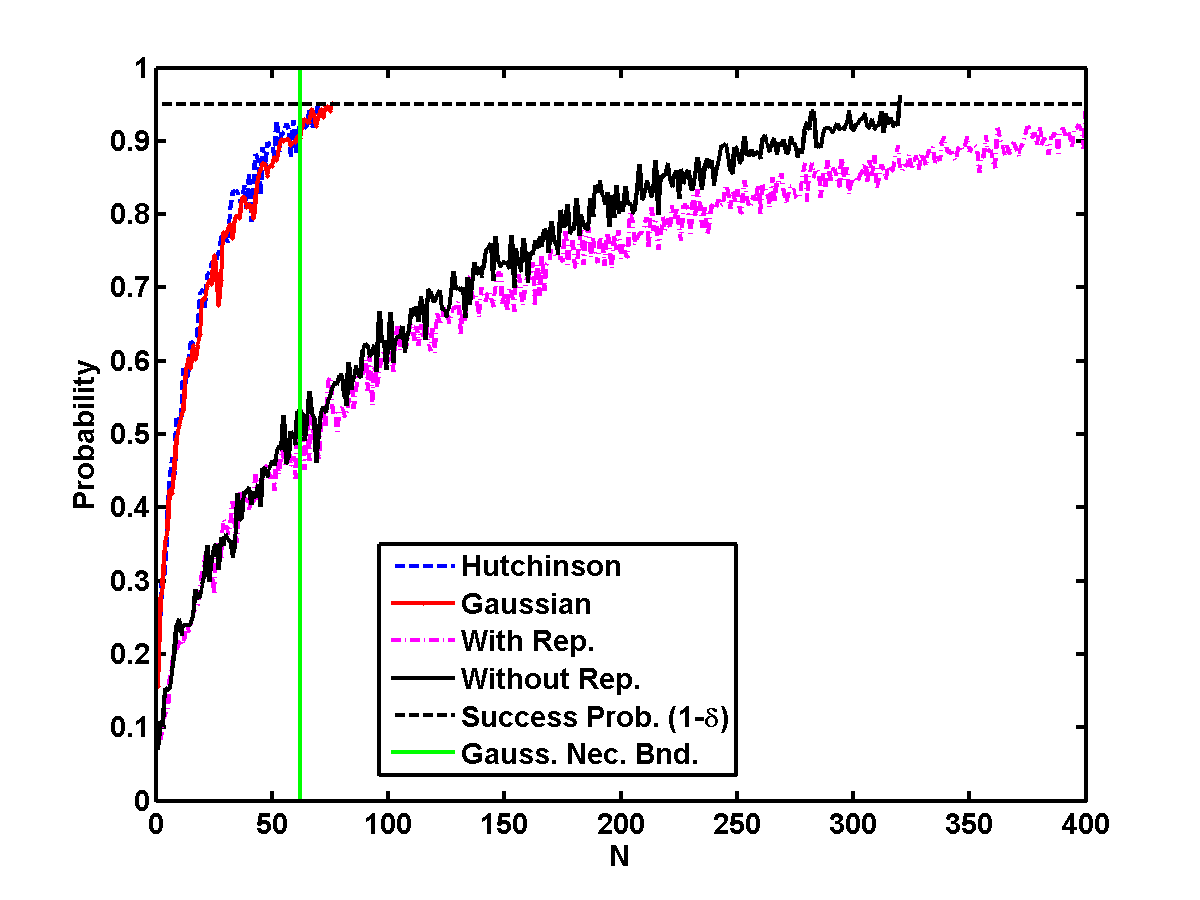}}
\subfigure[$\KH^{j}$ distribution ($\KH^{j} \leq 100$)]{
\includegraphics[scale=0.36]{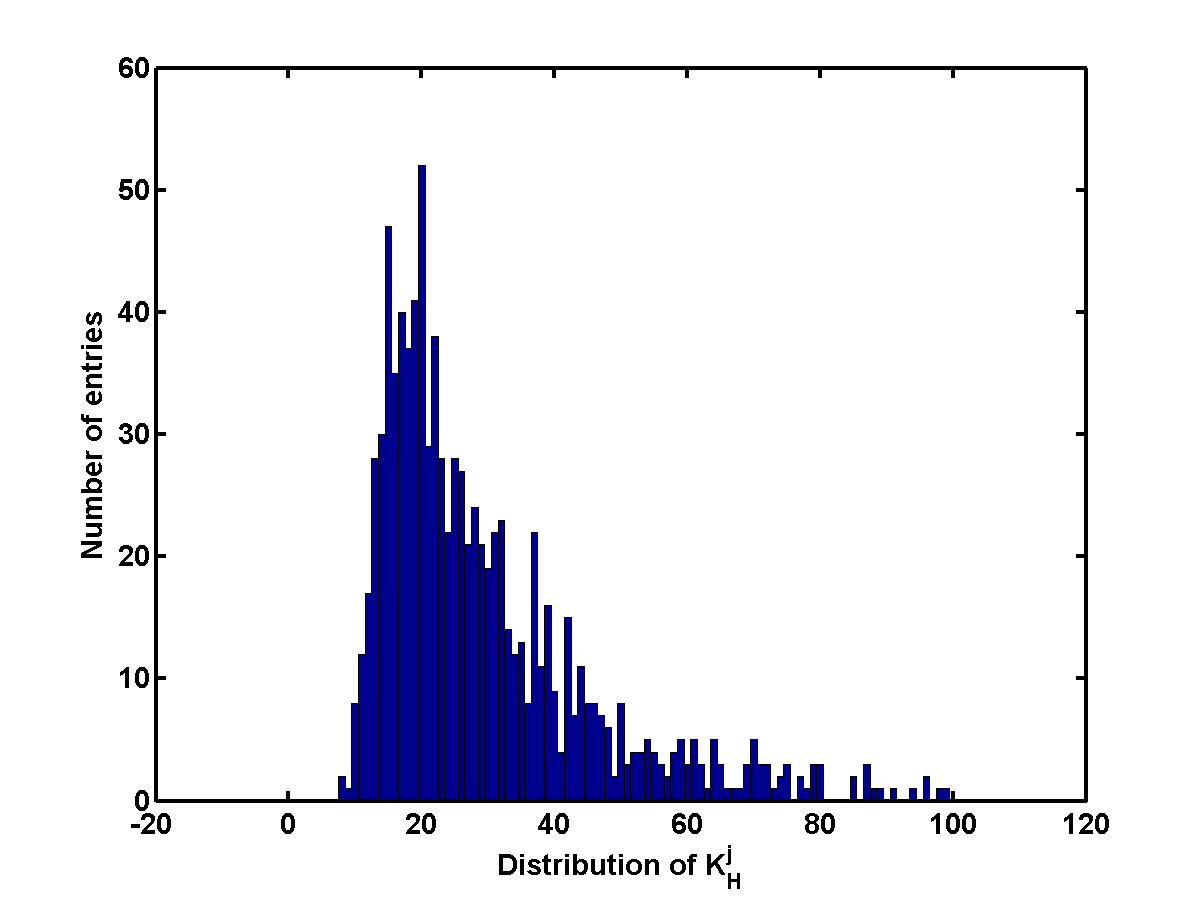}}
\subfigure[$\KU^{(i,j)}$ distribution]{
\includegraphics[scale=0.36]{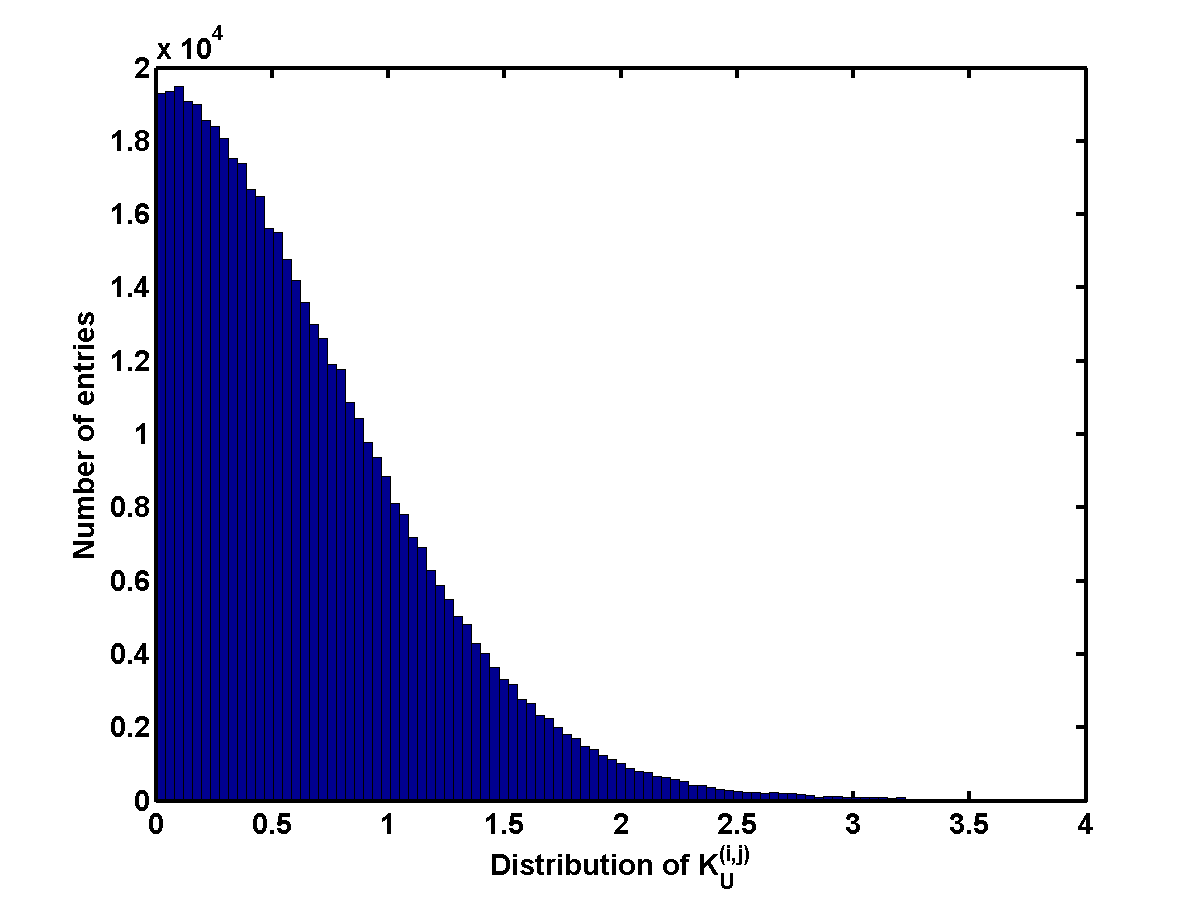}}
\subfigure[Eigenvalue distribution]{
\includegraphics[scale=0.36]{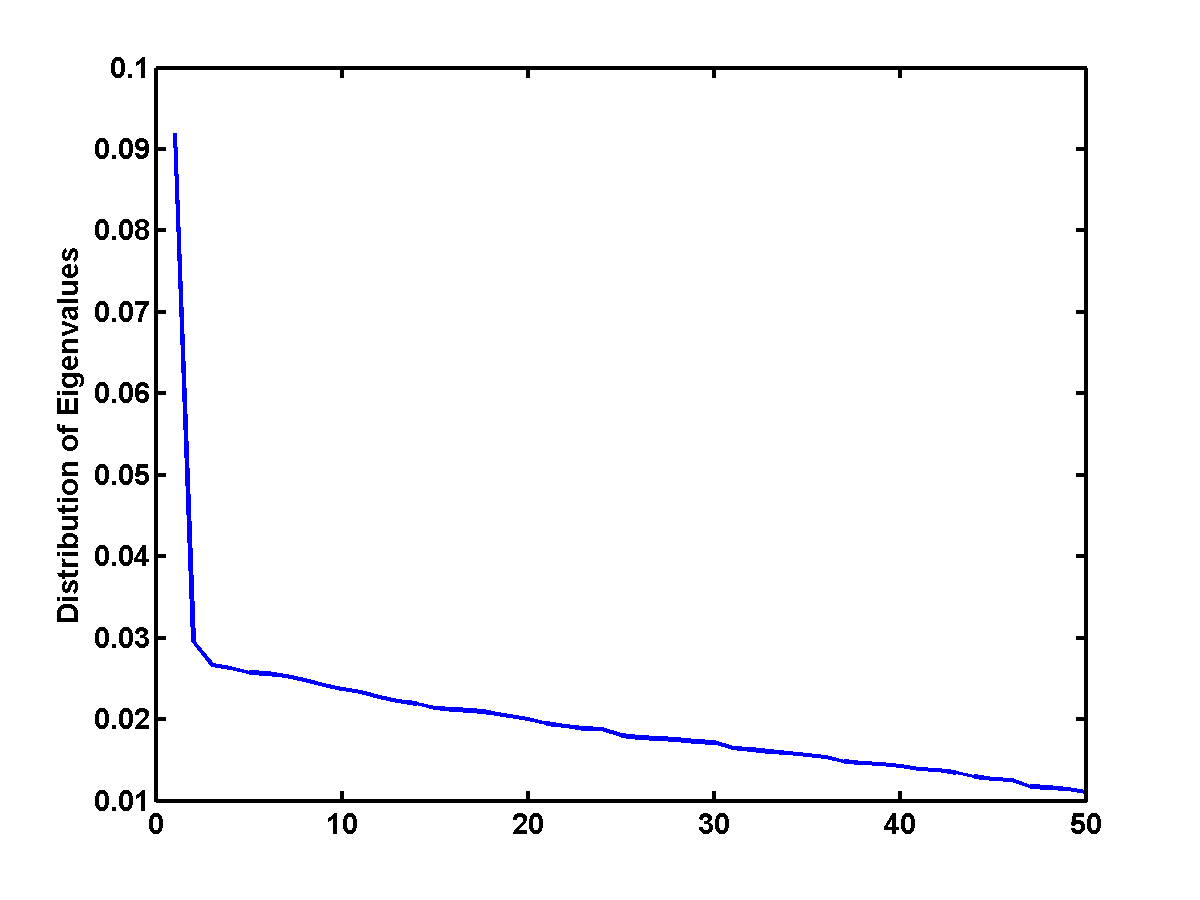}}
\caption{Example~\ref{exp_properties}. A sparse matrix {\tt (d = 0.1)} is formed using {\tt sprand}. 
Here $r=50, \KG = 0.0919,\, \KH = 11624.58$ and $\KU = 3.8823$.}
\label{exp_prop_fig04}
\end{figure}

\begin{figure}[htb]
\centering
\subfigure[Convergence Rate]{
\includegraphics[scale=0.36]{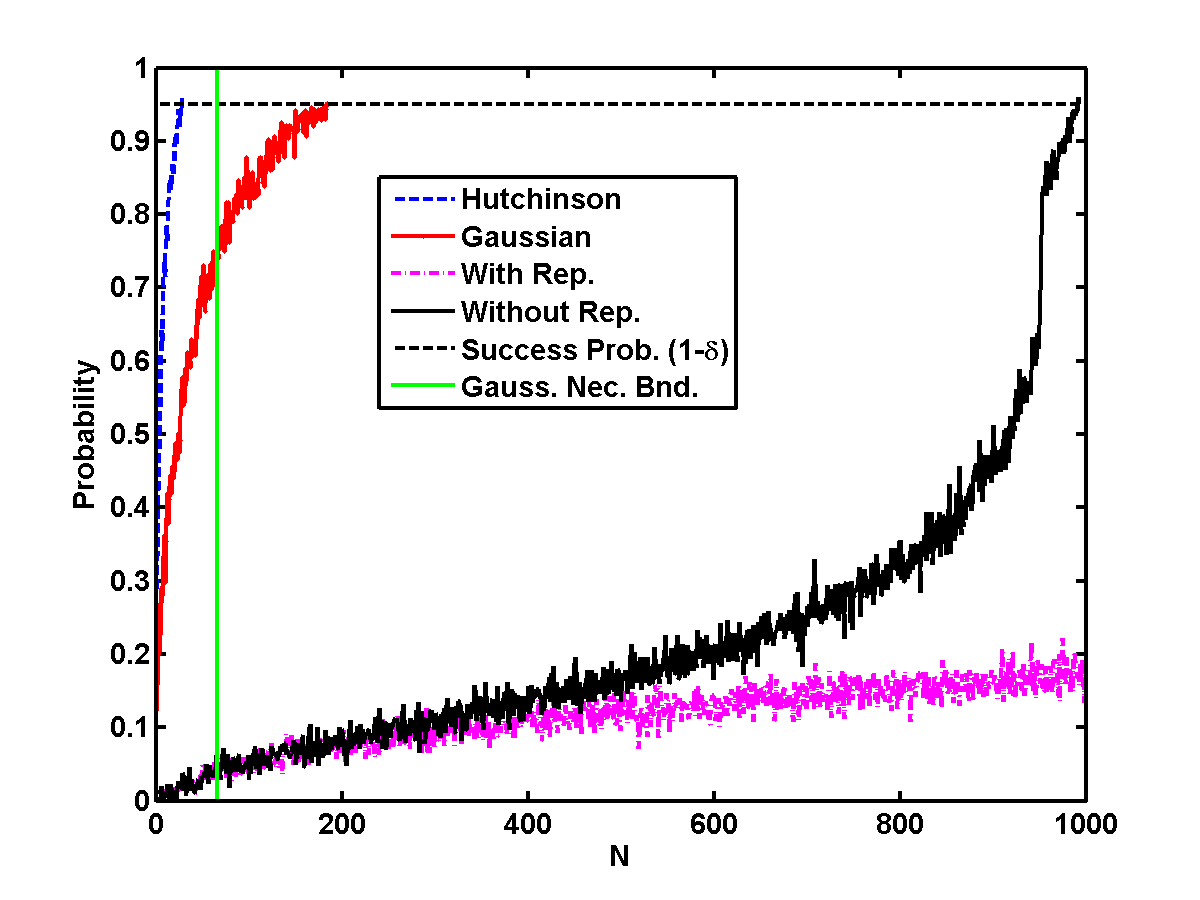}}
\subfigure[$\KH^{j}$ distribution ($\KH^{j} \leq 50$)]{
\includegraphics[scale=0.36]{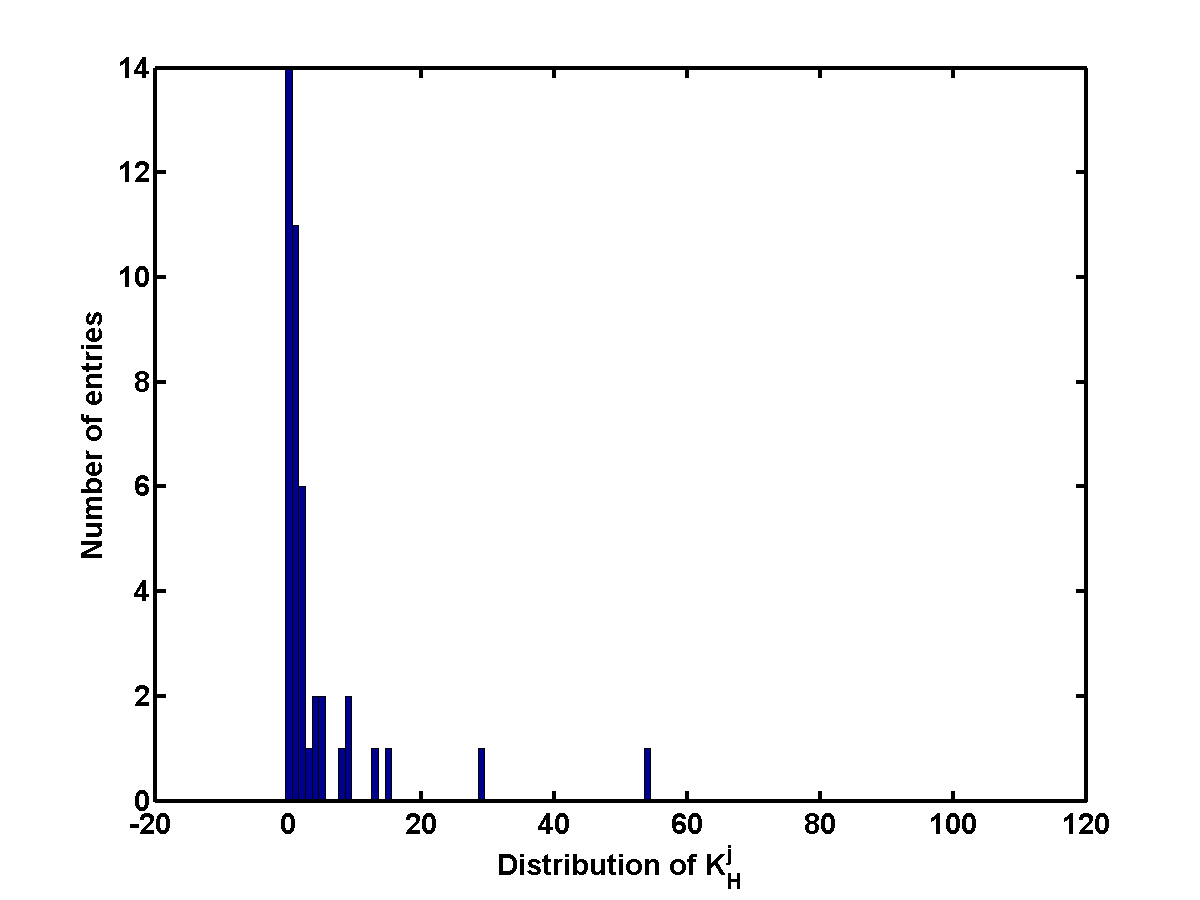}}
\subfigure[$\KU^{(i,j)}$ distribution]{
\includegraphics[scale=0.36]{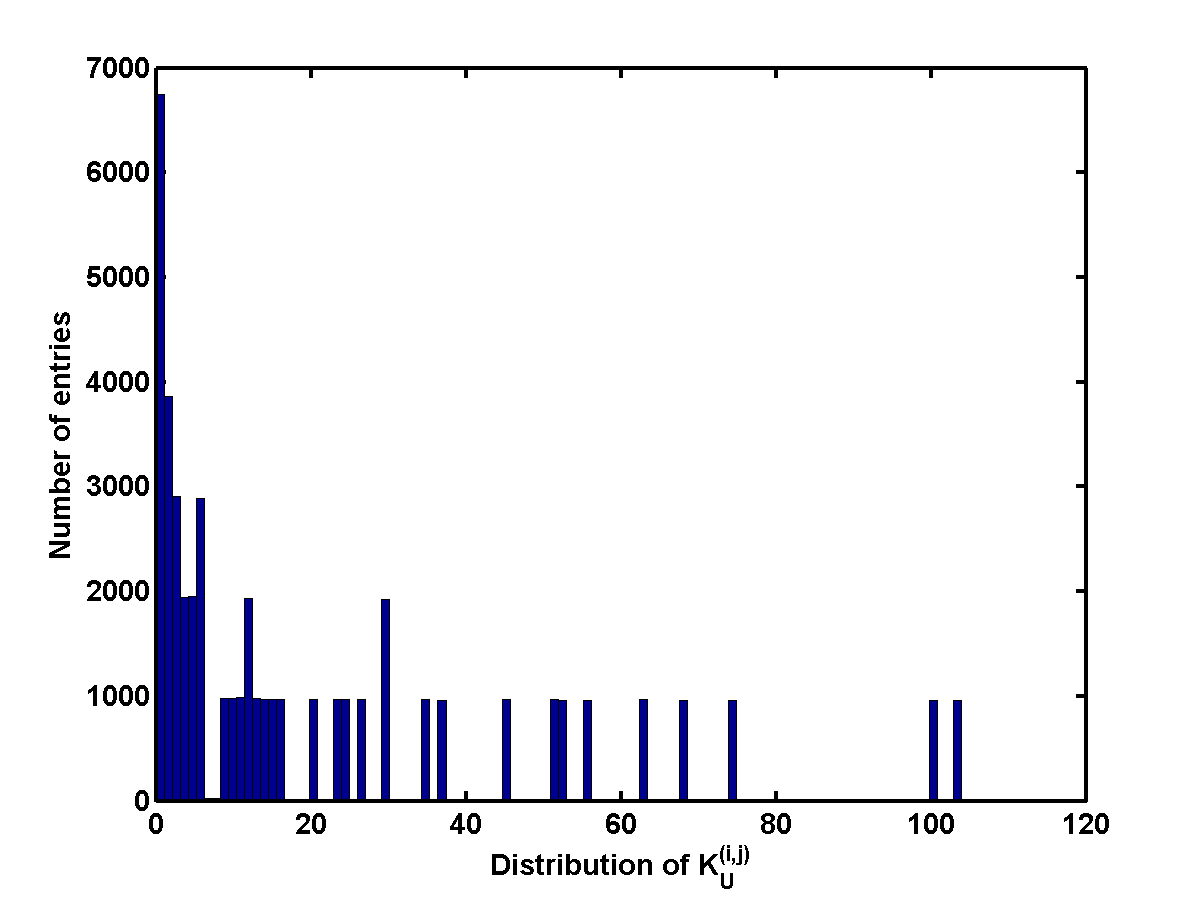}}
\subfigure[Eigenvalue distribution]{
\includegraphics[scale=0.36]{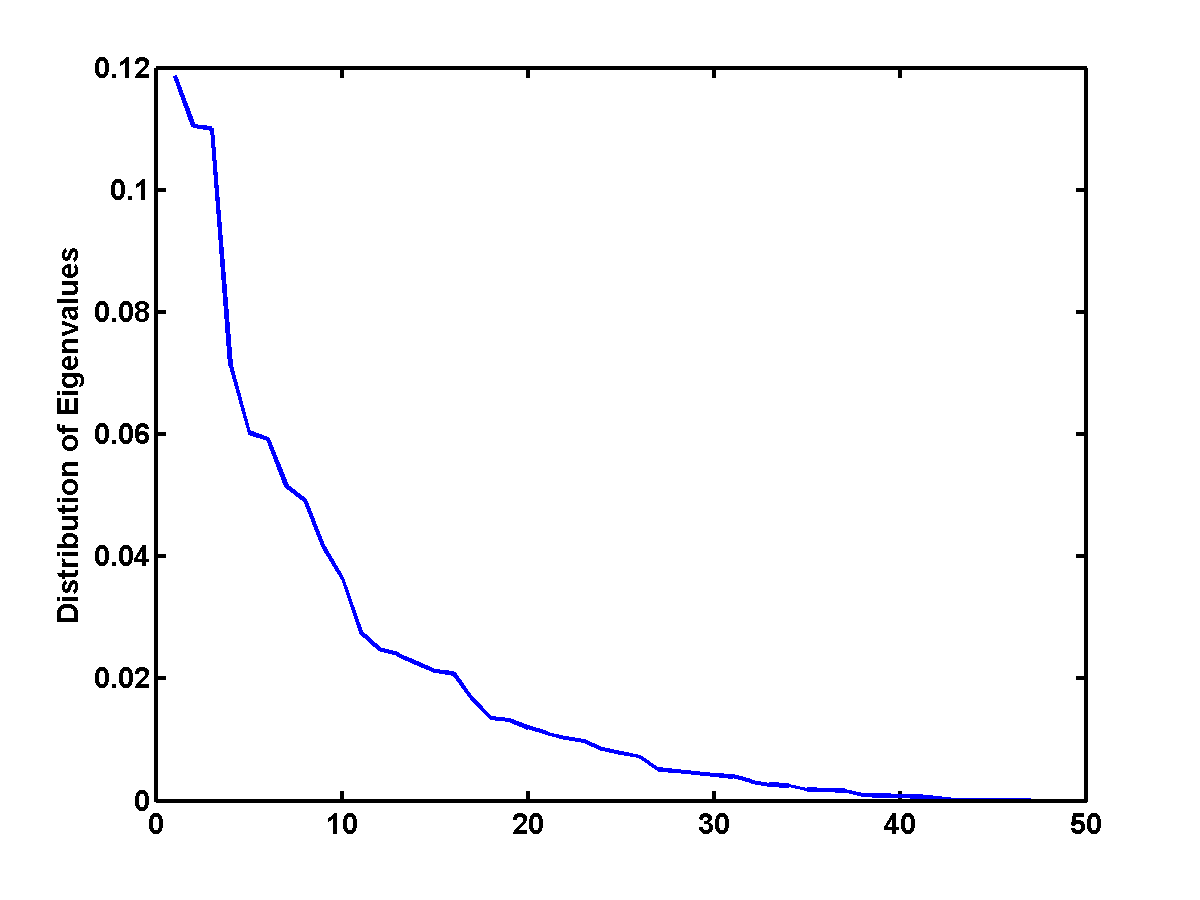}}
\caption{Example~\ref{exp_properties}. A very sparse matrix {\tt (d = 0.01)} is formed using {\tt sprandn}. 
Here $r=50, \, \KG = 0.1186, \, \KH = 8851.8$ and $\KU = 103.9593$.}
\label{exp_prop_fig05}
\end{figure}
  
By comparing Figures~\ref{exp_prop_fig03} and~\ref{exp_prop_fig04}, as well as~\ref{exp_prop_fig05} and~\ref{exp_prop_fig06},
we can see how not only the values of $\KH, \, \KG$ and $\KU$, but also the distribution of the quantities they maximize matters.
Note how the performance of both unit vector strategies is negatively affected with increasing 
average values of $\KU^{(i,j)}$'s. From the eigenvalue (or $\KG^{j}$)
distribution of the matrix, 
it can also be seen that the Gaussian estimator is heavily affected by the skewness of the distribution
of the eigenvalues (or $\KG^{j}$'s): 
given the same $r$ and $n$, as this eigenvalue distribution becomes increasingly uneven, the Gaussian method requires larger sample size. 


Note that comparing the performance of the methods on different matrices solely based on their values $\KH$, $\KG$ or $\KU$ can be misleading. 
This can be seen for instance by considering the performance of the Hutchinson method in Figures~\ref{exp_prop_fig03},~\ref{exp_prop_fig04},~\ref{exp_prop_fig05} 
and~\ref{exp_prop_fig06} and comparing their respective $\KH^{j}$ distributions as well as $\KH$ values.
Indeed, none of our $6$ sufficient bounds can be guaranteed to be generally tight.
As remarked also earlier, this is an artifact of the generality of the proved results.

Note also that rank and eigenvalue distribution of a matrix have no direct effect on the performance of the Hutchinson method: 
by Figures~\ref{exp_prop_fig05} and~\ref{exp_prop_fig06} it appears to only depend on the $\KH^{j}$ distribution. 
In these figures, one can observe that the Gaussian method is heavily affected by the low rank and the skewness of the eigenvalues. 
Thus, if the distribution of $\KH^{j}$'s is favourable to the Hutchinson method and yet the eigenvalue distribution is rather skewed, 
we can expect a significant difference between the performance of the Gaussian and Hutchinson methods. 
\end{example}

\section{Conclusions and further thoughts}
\label{concl}

In this article we have proved six sufficient bounds for the minimum sample size $N$ required
to reach, with probability $1-\delta$, an approximation for $tr(A)$ to within a relative tolerance $\veps$.
Two such bounds apply to each of the three estimators considered in Sections~\ref{hutch}, \ref{gauss}
and~\ref{randsamp}, respectively. In Section~\ref{gauss} we have also proved a necessary
bound for the Gaussian estimator.
These bounds have all been verified numerically through many examples, some of which are summarized
in Section~\ref{numer}.
\begin{figure}[htb]
\centering
\subfigure[Convergence Rate]{
\includegraphics[scale=0.36]{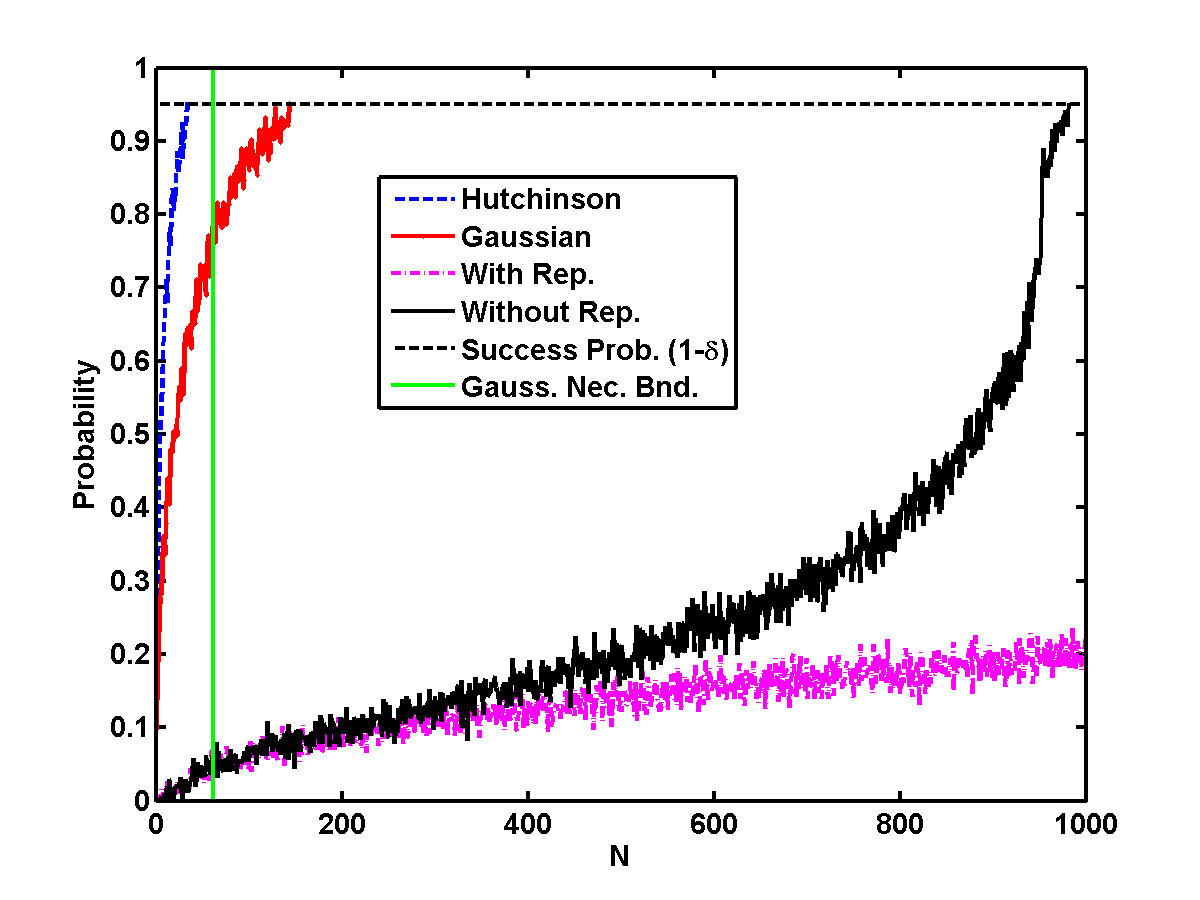}}
\subfigure[$\KH^{j}$ distribution ($\KH^{j} \leq 50$)]{
\includegraphics[scale=0.36]{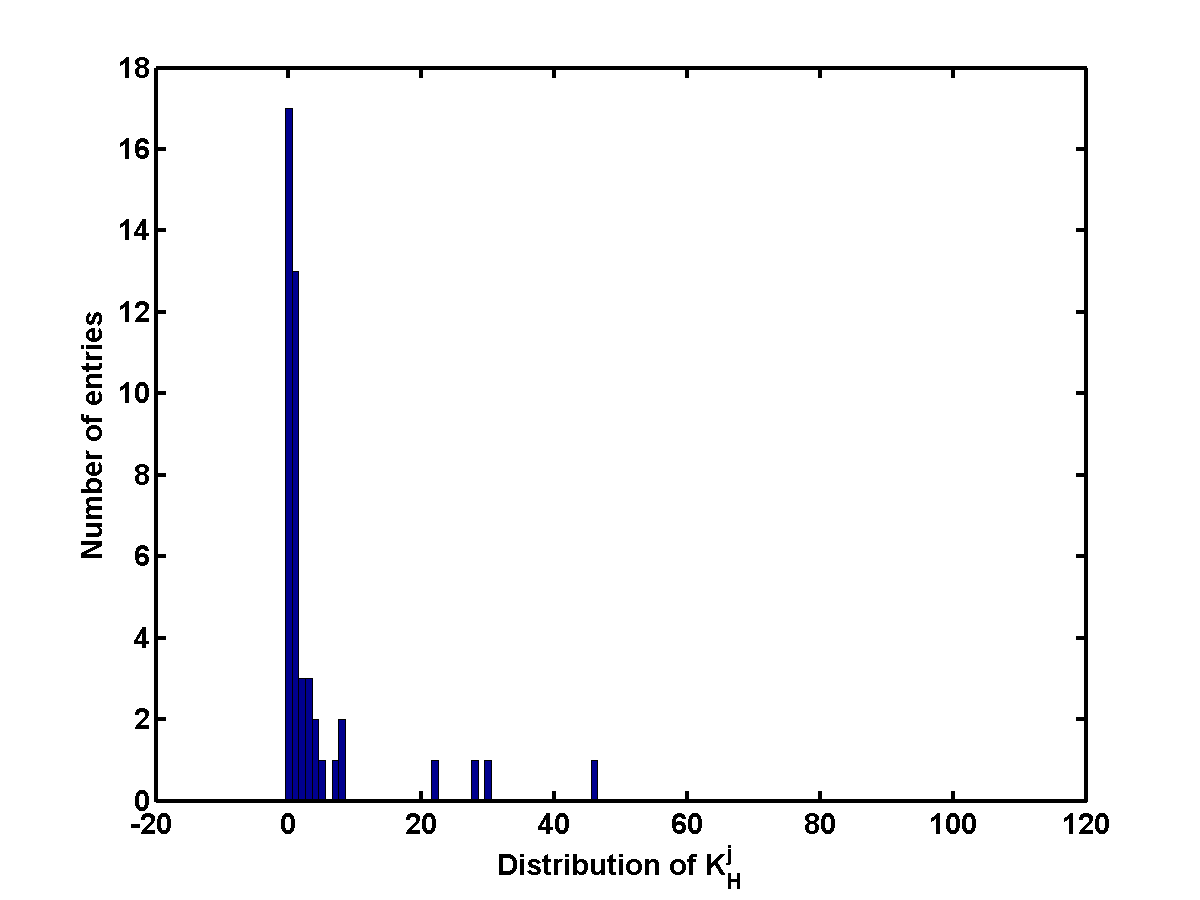}}
\subfigure[$\KU^{(i,j)}$ distribution]{
\includegraphics[scale=0.36]{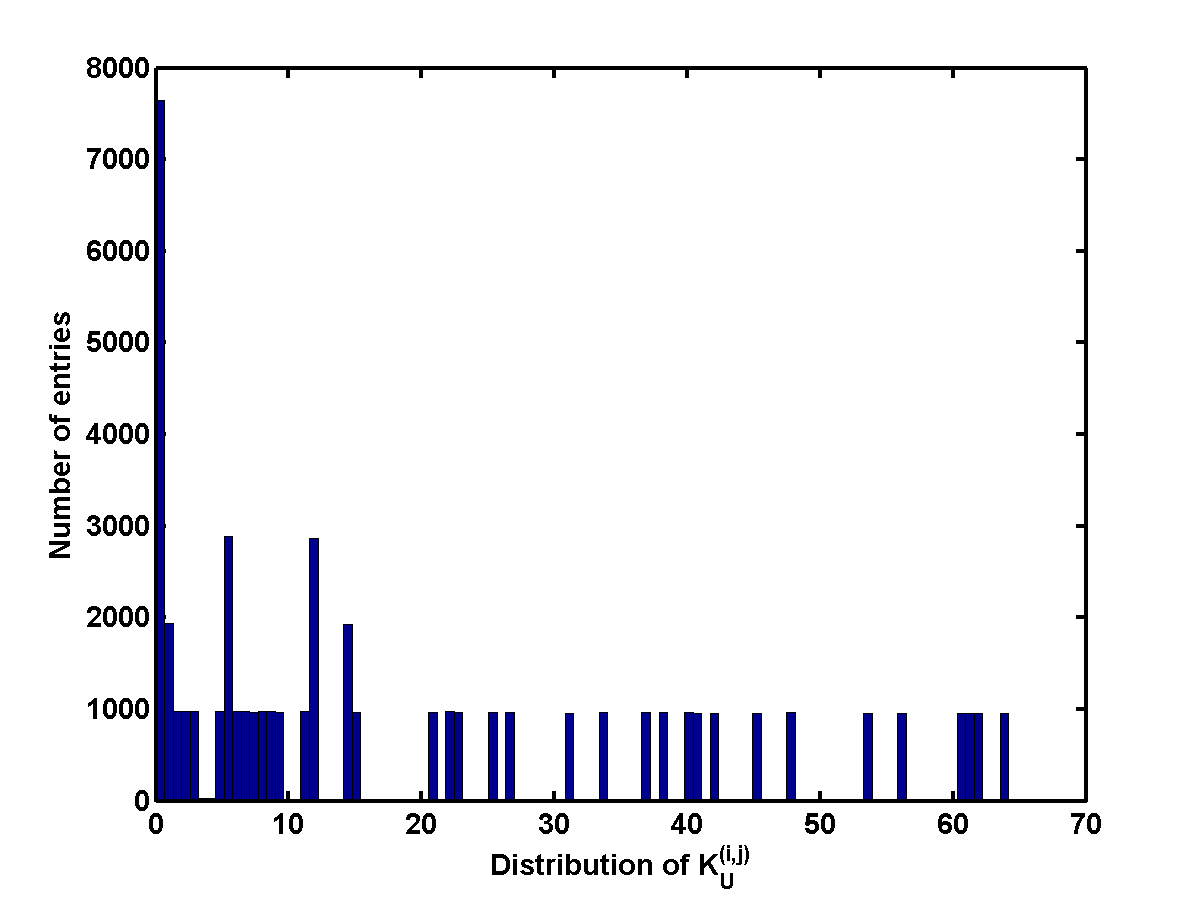}}
\subfigure[Eigenvalue distribution]{
\includegraphics[scale=0.36]{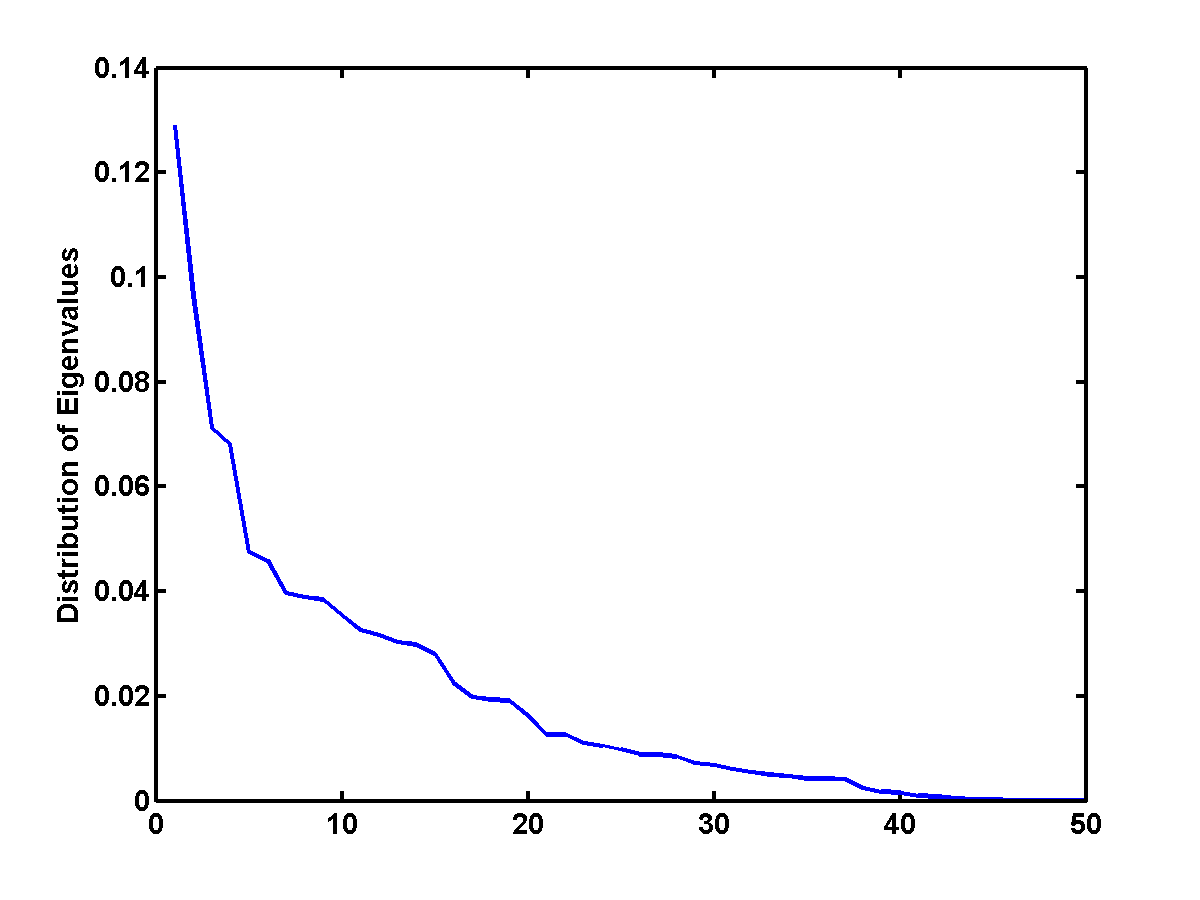}}
\caption{Example~\ref{exp_properties}. A very sparse matrix {\tt (d = 0.01)} is formed using {\tt sprand}.  
Here $r=50, \, \KG = 0.1290, \, \KH = 1611.34$ and $\KU = 64.1707$.}
\label{exp_prop_fig06}
\end{figure}


Two of these bounds, namely, \eqref{hutch_bd_01} for Hutchinson and \eqref{gauss_bd_01} for Gaussian, 
are immediately computable without knowing anything else about the SPSD matrix $A$.
In particular, they are independent of the matrix size $n$.
As such they may be very pessimistic. 
And yet, in some applications (for instance, in exploration geophysics)
where $n$ can be very large and $\veps$ need not be very small due to uncertainty, 
these bounds may indeed provide the comforting assurance that $N \ll n$ suffices
(say, $n$ is in the millions and $N$ in the thousands).
Generally, these two bounds have the same quality. 

The underlying objective in this work, which is to seek a small $N$ satisfying~\eqref{prob_tr},
is a natural one for many applications and follows that of other works.
But when it comes to comparing different methods, it is by no means the only 
performance indicator. For example, variance can also be considered as a ground to compare different methods. 
However, one needs to exercise caution to avoid basing the entire comparison solely on variance:
it is possible to generate examples where a linear combination of 
$\mathcal{X}^2$ random variables has smaller variance, yet higher tail probability.

The lower bound \eqref{gauss_bd_03} that is available only for the Gaussian
estimator may allow better prediction of the actual required $N$, in cases where the rank $r$ is known.
At the same time it also implies that the Gaussian estimator can be inferior
in cases where $r$ is small. The Hutchinson estimator does not enjoy a similar theory,
but empirically does not suffer from the same disadvantage either.
 
The matrix-dependent quantities $\KH$, $\KG$ and $\KU$, defined in \eqref{hutch_energy_dist}, 
\eqref{gauss_energy_dist} and \eqref{randsamp1_energy_dist}, respectively,
are not easily computable for any given implicit matrix $A$.
However, the results of Theorems~\ref{hutch_thm_02},~\ref{gauss_thm_01}
and~\ref{randsamp_thm} that depend on them can be more indicative than the general
bounds. In particular, examples where one method is clearly better than the others can be isolated
in this way.
At the same time, 
the sufficient conditions in Theorems~\ref{hutch_thm_02},~\ref{gauss_thm_01} and~\ref{randsamp_thm}, merely distinguish the types of matrices for which the respective methods are expected to be efficient, and make no claims regarding those matrices for which they are inefficient estimators. This is in direct  contrast with the necessary condition in Theorem~\ref{gauss_thm_02}.

It is certainly possible in some cases for the required $N$ to go over $n$. In this connection it
is important to always remember the deterministic method which obtains $tr(A)$ in $N$ applications
of unit vectors: if $N$ grows above $n$ in a particular stochastic setting then it may be best
to abandon ship and choose the safe, deterministic way. 

\input exline

\noindent{\bf Acknowledgment}
We thank our three anonymous referees for several valuable comments
which have helped to improve the text.\\
Part of this work was completed while the second author was visiting IMPA, Rio de Janeiro,
supported by a Brazilian ``Science Without Borders'' grant and hosted by Prof.
J. Zubelli. Thank you all. 


\bibliographystyle{plain}
\bibliography{biblio}

\end{document}

%% file: exline.tex
%
%
\vspace{0.5cm}